\documentclass[a4paper]{amsart}
\usepackage{amssymb}
\usepackage{verbatim}
\usepackage{epic}
\usepackage{eepic}
\usepackage[dvipdfmx]{graphicx}
\usepackage{tikz}
\usetikzlibrary{intersections,calc,arrows.meta,patterns}
\definecolor{refkey}{rgb}{1,0,0}
\definecolor{labelkey}{rgb}{0,0,1}
\theoremstyle{plain}
  \newtheorem{thm}{Theorem}[section]
  \newtheorem{lem}[thm]{Lemma}
  \newtheorem{cor}[thm]{Corollary}
  \newtheorem{prop}[thm]{Proposition}
  \newtheorem{conj}[thm]{Conjecture}
  \newtheorem*{conj*}{Conjecture}

  \newtheorem*{obs*}{Observation}
\theoremstyle{definition}
  \newtheorem{defn}[thm]{Definition}

\theoremstyle{remark}
  \newtheorem{rem}[thm]{Remark}

\newcommand{\N}{\mathbb{N}}
\newcommand{\Z}{\mathbb{Z}}
\newcommand{\C}{\mathbb{C}}
\newcommand{\R}{\mathbb{R}}
\newcommand{\Q}{\mathbb{Q}}
\newcommand{\Vol}{\operatorname{Vol}}

\newcommand{\CS}{\operatorname{CS}}
\newcommand{\CV}{\operatorname{CV}}

\newcommand{\Li}{\operatorname{Li}}

\renewcommand{\L}{\mathcal{L}}

\renewcommand{\sl}{\mathfrak{sl}}
\newcommand{\SL}{\mathrm{SL}}

\newcommand{\arccosh}{\operatorname{arccosh}}

\newcommand{\floor}[1]{\lfloor#1\rfloor}

\newcommand{\pic}[2]{\raisebox{-0.5\height}{\includegraphics[scale=#1]{#2.eps}}}

\renewcommand{\Re}{\operatorname{Re}}
\renewcommand{\Im}{\operatorname{Im}}

\renewcommand{\i}{\sqrt{-1}}

\makeatletter
\newcommand{\oset}[3][0ex]{%
  \mathrel{\mathop{#3}\limits^{
    \vbox to#1{\kern-1\ex@
    \hbox{$\scriptstyle#2$}\vss}}}}
\makeatother
\newcommand{\Rpath}{\oset{\frown}{\mathbb{R}}}
\numberwithin{equation}{section}
\allowdisplaybreaks
\begin{document}
\title[The colored Jones polynomial of the figure-eight knot]
{The colored Jones polynomial \\ of the figure-eight knot \\ and a quantum modularity}
\author{Hitoshi Murakami}
\address{
Graduate School of Information Sciences,
Tohoku University,
Aramaki-aza-Aoba 6-3-09, Aoba-ku,
Sendai 980-8579, Japan}
\email{hitoshi@tohoku.ac.jp}
\date{\today}
\begin{abstract}
We study the asymptotic behavior of the $N$-dimensional colored Jones polynomial of the figure-eight knot evaluated at $\exp\bigl((u+2p\pi\i)/N\bigr)$, where $u$ is a small real number and $p$ is a positive integer.
We show that it is asymptotically equivalent to the product of the $p$-dimensional colored Jones polynomial evaluated at $\exp\bigl(4N\pi^2/(u+2p\pi\i)\bigr)$ and a term that grows exponentially with growth rate determined by the Chern--Simons invariant.
This indicates a quantum modularity of the colored Jones polynomial.
\end{abstract}
\thanks{This work was supported by JSPS KAKENHI Grant Numbers JP22H01117, JP20K03601, JP20K03931.}
\maketitle
\section{Introduction}\label{sec:introduction}
Let $K$ be an oriented knot in the three-sphere $S^3$.
For a positive integer $N$, we denote by $J_N(K;q)$ the colored Jones polynomial associated with the irreducible $N$-dimensional representation of the Lie algebra $\mathfrak{sl}(2;\C)$.
Here we normalize $J_N(K;q)$ so that $J_N(U;q)=1$ for the unknot $U$.
\par
Let us consider an evaluation $J_N\left(K;e^{2\pi\i/N}\right)$.
It is well known that it coincides with Kashaev's invariant $\langle K\rangle_N$ \cite{Kashaev:MODPLA95,Murakami/Murakami:ACTAM12001}.
R.~Kashaev conjectured that his invariant grows exponentially as $N\to\infty$, and that its growth rate gives the hyperbolic volume of the knot complement when $K$ is a hyperbolic knot, that is, $S^3\setminus{K}$ possesses a (unique) complete hyperbolic structure with finite volume \cite{Kashaev:LETMP97}.
In \cite{Murakami/Murakami:ACTAM12001}, Kashaev's conjecture was generalized to any knot replacing the hyperbolic volume with simplicial volume (also known as Gromov's norm \cite{Gromov:INSHE82}).
\begin{conj}[Volume conjecture]
Let $K\subset S^3$ be any knot.
Then we have
\begin{equation*}
  \lim_{N\to\infty}
  \frac{1}{N}\log\bigl|J_N(K;e^{2\pi\i/N})\bigr|
  =
  \frac{1}{2\pi}\Vol(S^3\setminus{K}),
\end{equation*}
where $\Vol(S^3\setminus{K})$ is the simplicial volume of $S^3\setminus{K}$.
\end{conj}
\par
So far, Kashaev's conjecture is proved for the figure-eight knot by T.~Ekholm, and for knots with up to seven crossings \cite{Ohtsuki:QT2016,Ohtsuki/Yokota:MATPC2018,Ohtsuki:INTJM62017}.
The volume conjecture is proved for hyperbolic knots with up to seven crossings as above, for all the torus knots by Kashaev and O.~Tirkkonen \cite{Kashaev/Tirkkonen:ZAPNS2000}, for the Whitehead doubles of the torus knots by H.~Zheng \cite{Zheng:CHIAM22007}, and the $(2,2k+1)$-cable of the figure-eight knot by T.~Le and A.~Tran \cite{Le/Tran:JKNOT2010}.
\par
J.~Murakami, M.~Okamoto, T.~Takata, Y.~Yokota, and the author complexified Kashaev's conjecture as follows \cite[Conjecture~1.2]{Murakami/Murakami/Okamoto/Takata/Yokota:EXPMA02}:
\begin{conj}\label{conj:CVC}
For a hyperbolic knot $K$ in $S^3$, we have
\begin{equation*}
  J_N(K;e^{2\pi\i/N})
  \underset{N\to\infty}{\sim}
  \frac{N}{2\pi}\CV(K),
\end{equation*}
where $\CV(K):=\Vol(S^3\setminus{K})+\i\CS^{\mathrm{SO}(3)}(S^3\setminus{K})$ is the complex volume with $\CS^{\mathrm{SO}(3)}$ the $\mathrm{SO}(3)$ Chern--Simons invariant \cite{Meyerhoff:LMSLN112}.
\end{conj}
\par
For a hyperbolic knot $K\subset S^3$, let $\rho\colon\pi_1(S^3\setminus{K})\to\SL(2;\C)$ be an irreducible representation that is a small deformation of the holonomy representation corresponding to the complete hyperbolic structure.
Note that $\rho$ corresponds to an incomplete hyperbolic structure \cite{Thurston:2022}.
Up to conjugation, we may assume that $\rho$ sends the meridian of $K$ to $\begin{pmatrix}e^{u/2}&\ast\\0&e^{-u/2}\end{pmatrix}$ and the preferred longitude to $\begin{pmatrix}-e^{v(u)/2}&\ast\\0&-e^{-v(u)/2}\end{pmatrix}$ (see for example \cite{Neumann/Zagier:TOPOL1985}).
Associated with $u$, we can define the $\SL(2;\C)$ Chern--Simons invariant $\CS_{u,v(u)}(\rho)$ and the cohomological adjoint Reidemeister torsion $T_K(u)$.
See \cite[Chapter~5]{Murakami/Yokota:2018} for example.
Note that in \cite{Murakami/Yokota:2018} we define the \emph{homological} adjoint Reidemeister torsion (it is called the twisted Reidemeister torsion there).
So we need to take its inverse to define the cohomological torsion.
Note also that $\Vol(S^3\setminus{K})+\i\CS^{\mathrm{SO}(3)}(S^3\setminus{K})$ in Conjecture~\ref{conj:CVC} coincides with $\i\CS_{0,0}(\rho_0)$ for a hyperbolic knot $K$ with holonomy representation $\rho_0$.
\par
In \cite{Murakami/Yokota:JREIA2007}, Yokota and the author proved that for the figure-eight knot $E$, the limit $\lim_{N\to\infty}\frac{1}{N}\log J_{N}\left(E;e^{(u+2\pi\i)/N}\right)$ exists if the complex number $u$ is in a small neighborhood of $0$ (and not a rational multiple of $\pi\i$).
Moreover the limit determines the holomorphic function $f(u)$ introduced in \cite[Theorem~2]{Neumann/Zagier:TOPOL1985}.
In other words, the asymptotic behavior of $J_{N}(E;e^{(u+2\pi\i)/N})$ determines the $\SL(2;\C)$ Chern--Simons invariant associated with $u$.
\par
For a general hyperbolic knot $K$, the following conjecture was proposed in \cite{Murakami:JTOP2013} (see also \cite{Dimofte/Gukov:Columbia,Gukov/Murakami:FIC2008}).
\begin{conj}\label{conj:PVC}
Let $K\subset S^3$ be a hyperbolic knot.
Then there exists a neighborhood $U\subset\C$ of $0$ such that if $u\in U\setminus\pi\i\Q$, then we have
\begin{multline*}
  J_{N}\left(K;e^{(u+2\pi\i)/N}\right)
  \\
  \underset{N\to\infty}{\sim}
  \frac{\sqrt{-\pi}}{2\sinh(u/2)}
  T_K(u)^{1/2}\left(\frac{N}{u+2\pi\i}\right)^{1/2}
  \exp\left(\frac{N\times S_K(u)}{u+2\pi\i}\right),
\end{multline*}
where $T_K(u)$ is the cohomological adjoint Reidemeister torsion, and $\CS_{u,v(u)}(\rho)=S_K(u)-u\pi\i-\frac{1}{4}uv(u)$ is the Chern--Simons invariant, both associated with $u$.
\end{conj}
In \cite{Murakami:JTOP2013}, we proved that the conjecture is true for the figure-eight knot and a positive real number $u<\arccosh(3/2)$.
\par
In this paper, we study the colored Jones polynomial of the figure-eight knot evaluated at $q=\exp\bigl((u+2p\pi\i)/N\bigr)$ for a real number $u$ with $0<u<\arccosh(3/2)$ and a positive integer $p$.
We will show
\begin{thm}\label{thm:main}
Let $E$ be the figure-eight knot and put $\xi:=u+2p\pi\i$.
Then we have
\begin{multline}\label{eq:main}
  J_N\left(E;e^{\xi/N}\right)
  \\
  =
  \frac{\sqrt{-\pi}}{2\sinh(u/2)}T_E(u)^{1/2}
  J_{p}\left(E;e^{4N\pi^2/\xi}\right)
  \left(\frac{N}{\xi}\right)^{1/2}
  e^{\frac{N}{\xi}\times S_E(u)}
  \bigl(1+O(N^{-1})\bigr)
\end{multline}
as $N\to\infty$, where we put
\begin{align*}
  S_E(u)
  &:=
  \Li_2\left(e^{-u-\varphi(u)}\right)-\Li_2\left(e^{-u+\varphi(u)}\right)
  +
  u\bigl(\varphi(u)+2\pi\i\bigr),
  \\
  T_E(u)
  &:=
  \frac{2}{\sqrt{(2\cosh{u}+1)(2\cosh{u}-3)}}.
\end{align*}
Here $\Li_2(z):=-\int_{0}^{z}\frac{\log(1-w)}{w}\,dw$ is the dilogarithm function and we put
\begin{equation*}
  \varphi(u)
  :=
  \log
  \left(
    \cosh{u}-\frac{1}{2}-\frac{1}{2}\sqrt{(2\cosh{u}+1)(2\cosh{u}-3)}
  \right).
\end{equation*}
\end{thm}
\begin{rem}
The case where $p=1$ was proved in \cite{Murakami:JTOP2013}.
\end{rem}
\begin{rem}
When $p=0$, the author proved that for the figure-eight knot $E$, $J_N\left(E;e^{u/N}\right)$ converges to $1/\Delta(E;e^{u})$, where $\Delta(K;t)$ is the Alexander polynomial of a knot $K$ normalized so that $\Delta(K;t)=\Delta(K;t^{-1})$ and $\Delta(U;t)=1$ \cite{Murakami:JPJGT2007}.
Soon after, it was generalized by S.~Garoufalidis and T.~Le to any knot in $S^3$.
See \cite{Garoufalidis/Le:2005,Garoufalidis/Le:aMMR,Garoufalidis/Le:GEOTO2011}.
\end{rem}
As a corollary we have the following asymptotic equivalence.
\begin{cor}
We have
\begin{equation}\label{eq:J_p/J_1}
  \frac{J_N\left(E;e^{\xi/N}\right)}{J_p\left(E;e^{4N\pi^2/\xi}\right)}
  \underset{N\to\infty}{\sim}
  \frac{\sqrt{-\pi}}{2\sinh(u/2)}T_E(u)^{1/2}
  \left(\frac{N}{\xi}\right)^{1/2}
  e^{\frac{N}{\xi}\times S_E(u)}.
\end{equation}
\end{cor}
This indicates a quantum modularity for the colored Jones polynomial.
\begin{conj*}[Conjecture~\ref{conj:QMCCJ}]
Let $K$ be a hyperbolic knot.
For a small complex number $u$ that is not a rational multiple of $\pi\i$, and positive integers $p$ and $N$, put $\xi:=u+2p\pi\i$ and $X:=\frac{2N\pi\i}{\xi}$.
Then for any $\eta=\begin{pmatrix}a&b\\c&d\end{pmatrix}\in\SL(2;\Z)$ with $c>0$, the following asymptotic equivalence holds.
\begin{equation*}
  \frac{J_{cN+dp}\left(K;e^{2\pi\i\eta(X)}\right)}{J_{p}\left(K;e^{2\pi\i X}\right)}
  \underset{N\to\infty}{\sim}
  C_{K,\eta}(u)\frac{\sqrt{-\pi}}{2\sinh(u/2)}
  \left(\frac{T_K(u)}{\hbar_{\eta}(X)}\right)^{1/2}
  \exp\left(\frac{S_K(u)}{\hbar_{\eta}(X)}\right)
\end{equation*}
for $C_{K,\eta}(u)\in\C$ that does not depend on $p$, where we put $\eta(X):=\frac{aX+b}{cX+d}$ and $\hbar_{\eta}(X):=\frac{2c\pi\i}{cX+d}$.
\end{conj*}
Compare it with Zagier's quantum modularity conjecture for Kashaev's invariant \cite{Zagier:2010}.
\begin{conj*}[Conjecture~\ref{conj:modular_Zagier}]
Let $K$, $\eta$, $N$, and $p$ as above.
If we put $X_0:=\frac{N}{p}$, the following holds.
\begin{equation*}
  \frac{J_{cN+dp}\left(K;e^{2\pi\i\eta(X_0)}\right)}{J_{p}\left(K;e^{2\pi\i X_0}\right)}
  \\
  \underset{N\to\infty}{\sim}
  C_{K,\eta}
  \left(\frac{2\pi}{\hbar_{\eta}(X_0)}\right)^{3/2}
  \exp\left(\frac{\i\CV(K)}{\hbar_{\eta}(X_0)}\right),
\end{equation*}
where $C_{K,\eta}$ is a complex number depending only on $\eta$ and $K$.
\end{conj*}
\par
The paper is organized as follows.
\par
In Section~\ref{sec:preliminaries}, we define the colored Jones polynomial and introduce a quantum dilogarithm.
We express the colored Jones polynomial as a sum of the quantum dilogarithms assuming $(p,N)=1$ in Section~\ref{sec:summation}.
In Section~\ref{sec:approximation}, we approximate it by using the dilogarithm function by using the fact that the quantum dilogarithm converges to the dilogarithm.
We use the Poisson summation formula \'a la Ohtsuki \cite{Ohtsuki:QT2016} to replace the sum with an integral in Section~\ref{sec:Poisson}.
In Section~\ref{sec:proof}, we prove the main theorem (Theorem~\ref{thm:main}).
We discuss a quantum modularity of the colored Jones polynomial in Section~\ref{sec:modular}.
Section~\ref{sec:lemmas} is devoted to proofs of lemmas used in the other sections.
In Appendix, we calculate the colored Jones polynomial in the case where $(p,N)\ne1$.
\section{Preliminaries}\label{sec:preliminaries}
Let $J_N(K;q)$ be the $N$-dimensional colored Jones polynomial of $K\subset S^3$ associated with the $N$-dimensional irreducible representation of the Lie algebra $\sl_2(\C)$, where $N$ is a positive integer and $q$ is a complex parameter \cite{Kirillov/Reshetikhin:1989,Reshetikhin/Turaev:COMMP90,Kirby/Melvin:INVEM1991}.
It is normalized so that $J_N(U;q)=1$ for the unknot $U$.
In particular, $J_2(K;q)$ is (a version) of the original Jones polynomial \cite{Jones:BULAM31985}.
More precisely, $J_2(K;q)$ satisfies the following skein relation:
\begin{equation*}
  q
  J_2\left(\raisebox{-1.8mm}{\includegraphics[scale=0.1]{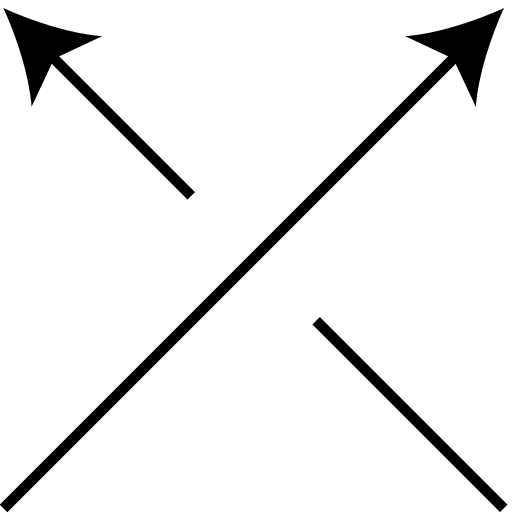}};q\right)
  -
  q^{-1}
  J_2\left(\raisebox{-1.8mm}{\includegraphics[scale=0.1]{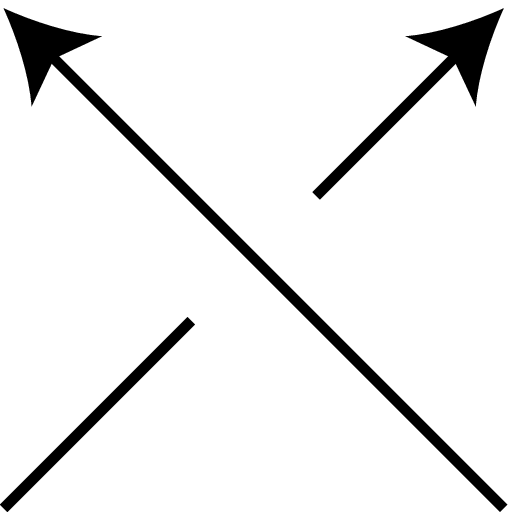}};q\right)
  =
  \left(q^{1/2}-q^{-1/2}\right)
  J_2\left(\raisebox{-1.8mm}{\includegraphics[scale=0.1]{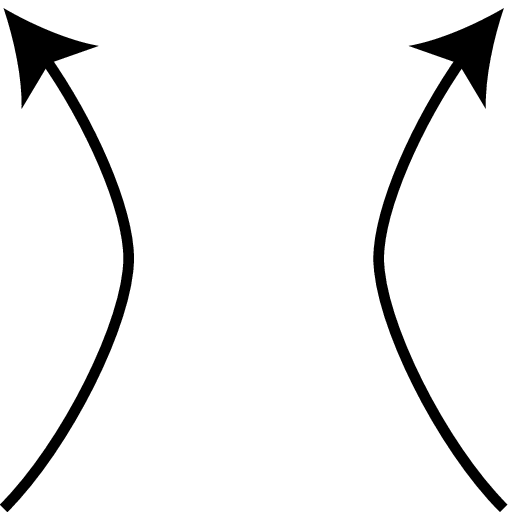}};q\right).
\end{equation*}
\par
K.~Habiro \cite[P.~36 (1)]{Habiro:SURIK2000} (see also \cite[Theorem~5.1]{Masbaum:ALGGT12003}) and T.~Le \cite[1.2.2.~Example, P.~129]{Le:TOPOA2003} gave a simple formula for the colored Jones polynomial of the figure-eight knot $E$ as follows:
\begin{align}\label{eq:J}
  J_{N}(E;q)
  &=
  \sum_{k=0}^{N-1}\prod_{l=1}^{k}
  \left(q^{(N+l)/2}-q^{-(N+l)/2}\right)\left(q^{(N-l)/2}-q^{-(N-l)/2}\right)
  \\
  &=
  \sum_{k=0}^{N-1}q^{-kN}
  \prod_{l=1}^{k}\left(1-q^{N+l}\right)\left(1-q^{N-l}\right).
\end{align}
For a real number $u$ with $0<u<\kappa:=\arccosh(3/2)=0.962424\dots$, and a positive integer $p$, we put $\xi:=u+2p\pi\i$.
Then we have
\begin{equation}\label{eq:J_xi}
  J_N\left(E;e^{\xi/N}\right)
  =
  \sum_{k=0}^{N-1}e^{-k\xi}
  \prod_{l=1}^{k}\left(1-e^{(N+l)\xi/N}\right)\left(1-e^{(N-l)\xi/N}\right).
\end{equation}
\par
We want to replace the products in \eqref{eq:J_xi} with some values of a continuous function.
To do that we introduce a so-called quantum dilogarithm following \cite{Faddeev:LETMP1995}.
\par
Put $\Rpath:=(-\infty,-1]\cup\{z\in\C\mid|z|=1,\Im{z}\ge0\}\cup[1,\infty)$ and orient it from left to right.
We consider the integral $\int_{\Rpath}\frac{e^{(2z-1)x}}{x\sinh(x)\sinh(\gamma x)}\,dx$, where $\gamma:=\frac{\xi}{2N\pi\i}$.
\begin{lem}\label{lem:T_N_convergence}
The integral $\int_{\Rpath}\frac{e^{(2z-1)x}}{x\sinh(x)\sinh(\gamma x)}\,dx$ converges if $-p/(2N)<\Re{z}<1+p/(2N)$.
\end{lem}
A proof is given in \S~\ref{sec:lemmas}.
Note that the poles of the integrand is
\begin{equation*}
  \{x\in\C\mid x=k\pi\i\;\;\text{($k\in\Z$)}
  \}
  \cup
  \{x\in\C\mid x=l\pi\i/\xi\;\;\text{($l\in\Z$)}\}
\end{equation*}
and so $\Rpath$ avoids the poles.
\par
We define
\begin{equation*}
  T_N(z)
  :=
  \frac{1}{4}
  \int_{\Rpath}\frac{e^{(2z-1)x}}{x\sinh(x)\sinh(\gamma x)}\,dx.
\end{equation*}
\par
We also consider three related integrals $\int_{\Rpath}\frac{e^{(2z-1)x}}{x^k\sinh(x)}\,dx$ ($k=0,1,2$), which converge if $0<\Re{z}<1$ by similar reasons to Lemma~\ref{lem:T_N_convergence}.
\begin{defn}\label{def:L_k}
We put
\begin{align*}
  \L_0(z)
  &:=
  \int_{\Rpath}\frac{e^{(2z-1)x}}{\sinh(x)}\,dx,
  \\
  \L_1(z)
  &:=
  \frac{-1}{2}\int_{\Rpath}\frac{e^{(2z-1)x}}{x\sinh(x)}\,dx,
  \\
  \L_2(z)
  &:=
  \frac{\pi\i}{2}\int_{\Rpath}\frac{e^{(2z-1)x}}{x^2\sinh(x)}\,dx
\end{align*}
for $z$ with $0<\Re{z}<1$.
\end{defn}
Their derivatives are given as follows.
\begin{align*}
  \frac{d\,\L_2(z)}{d\,z}
  &=
  -2\pi\i\L_1(z),
  \\
  \frac{d\,\L_1(z)}{d\,z}
  &=
  -\L_0(z).
\end{align*}
\par
We also have the following lemma.
\begin{lem}\label{lem:L_k}
If $0<\Re{z}<1$, then we have
\begin{align*}
  \L_0(z)
  &=
  \frac{-2\pi\i}{1-e^{-2\pi\i z}},
  \\
  \L_1(z)
  &=
  \log\left(1-e^{2\pi\i z}\right),
  \\
  \L_2(z)
  &=
  \Li_2\left(e^{2\pi\i z}\right).
\end{align*}
Here we use the branch of $\log{w}$ so that $-\pi<\Im\log{w}\le\pi$ and $\Li_2(w)$ has branch cut at $(1,\infty)$.
\end{lem}
\begin{proof}
As \cite[Lemma~2.5]{Murakami/Tran:2021}, we can prove the following equalities:
\begin{align*}
  \L_0(z)
  &=
  \frac{-2\pi\i}{1-e^{-2\pi\i z}},
  \\
  \L_1(z)
  &=
  \begin{cases}
    \log\left(1-e^{2\pi\i z}\right)             &\text{if $\Im{z}\ge0$}, \\
    \pi\i(2z-1)+\log\left(1-e^{-2\pi\i z}\right)&\text{if $\Im{z}<0$},
  \end{cases}
  \\
  \L_2(z)
  &=
  \begin{cases}
    \Li_2\left(e^{2\pi\i z}\right)                            &\text{if $\Im{z}\ge0$}, \\
    \frac{\pi^2}{3}(6z^2-6z+1)-\Li_2\left(e^{-2\pi\i z}\right)&\text{if $\Im{z}<0$}.
  \end{cases}
\end{align*}
So we need to prove the lemma for the case where $\Im{z}<0$.
\par
There is nothing to prove for $\L_0(z)$.
\par
If $0<\Re{z}<1$, then using the identity (see for example \cite{Maximon:2003})
\begin{equation}\label{eq:dilog}
  \Li_2(w^{-1})
  =
  -\Li_2(w)-\frac{\pi^2}{6}-\frac{1}{2}\bigl(\log(-w)\bigr)^2,
\end{equation}
we have
\begin{align*}
  \Li_2\left(e^{-2\pi\i z}\right)
  &=
  -\Li_2\left(e^{2\pi\i z}\right)-\frac{\pi^2}{6}
  -\frac{1}{2}\bigl(2\pi\i z-\pi\i\bigr)^2
  \\
  &=
  -\Li_2\left(e^{2\pi\i z}\right)
  +2\pi^2z^2+\frac{\pi^2}{3},
  \notag
\end{align*}
where we use the fact that $0<\Im(2\pi\i z)<2\pi$.
Therefore we have
\begin{equation*}
  \L_2(z)
  =
  -\Li_2\left(e^{-2\pi\i z}\right)+\frac{\pi^2}{3}\bigl(6z^2-6z+1\bigr)
  =
  \Li_2\left(e^{2\pi\i z}\right)
\end{equation*}
as required.
\par
As for $\L_1(z)$, since $\log\left(e^{\pi\i(2z-1)}\right)=2\pi\i z-\pi\i$, we have
\begin{align*}
  \log\left(1-e^{-2\pi\i z}\right)+\pi\i(2z-1)
  &=
  \log\left(1-e^{-2\pi\i z}\right)+\log\left(e^{\pi\i(2z-1)}\right)
  \\
  &=
  \log\left(1-e^{2\pi\i z}\right),
\end{align*}
completing the proof.
\end{proof}
\par
We can prove that $T_N(z)$ converges to $\frac{N}{\xi}\Li_2\left(e^{2\pi\i z}\right)$.
More precisely we have
\begin{lem}\label{lem:T_N_L_2}
For any positive real number $M$ and a sufficiently small positive real number $\nu$, we have
\begin{equation*}
  T_N(z)
  =
  \frac{N}{\xi}\Li_2\left(e^{2\pi\i z}\right)+O(1/N)
\end{equation*}
as $N\to\infty$ in the region
\begin{equation*}
  \{z\in\C\mid\nu\le\Re{z}\le1-\nu,|\Im{z}|\le M\}.
\end{equation*}
In particular $T_N(z)$ uniformly converges to $\frac{N}{\xi}\Li_2\left(e^{2\pi\i z}\right)$ in the region above.
\end{lem}
A proof is also given in \S~\ref{sec:lemmas}.
\par
The following lemma is essential in the paper.
Put $E_N(z):=e^{T_N(z)}$.
\begin{lem}\label{lem:q_dilog}
If $0<\Re{z}<1$, then we have
\begin{equation*}
  \frac{E_N(z-\gamma/2)}{E_N(z+\gamma/2)}
  =
  1-e^{2\pi\i z}.
\end{equation*}
\end{lem}
\begin{proof}
Recalling that $\gamma=\frac{\xi}{2N\pi\i}$, we have
\begin{align*}
  T_N(z-\gamma/2)-T_N(z+\gamma/2)
  &=
  \frac{1}{4}
  \int_{\Rpath}
  \frac{e^{(2z-\gamma-1)x}-e^{(2z+\gamma-1)x}}
       {x\sinh(x)\sinh(\gamma x)}\,dx
  \\
  &=
  -
  \int_{\Rpath}
  \frac{e^{(2z-1)x}}{2x\sinh(x)}\,dx
  =
  \L_1(z).
\end{align*}
Taking the exponentials of both sides, the lemma follows from Lemma~\ref{lem:L_k}.
\end{proof}
As a corollary, we have
\begin{cor}\label{cor:q_dilog}
Let $n$ be an integer.
If $nN/p<j<(n+1)N/p$, we have
\begin{align*}
  \frac{E_N\bigl((j-1/2)\gamma-n\bigr)}{E_N\bigl((j+1/2)\gamma-n\bigr)}
  &=
  1-e^{2j\gamma\pi\i},
  \\
  \intertext{and}
  \frac{E_N\bigl(n+1-(j+1/2)\gamma\bigr)}{E_N\bigl(n+1-(j-1/2)\gamma\bigr)}
  &=
  1-e^{-2j\gamma\pi\i}.
\end{align*}
\end{cor}
\begin{proof}
Since $\Re\gamma=p/N$, we have $0<\Re(j\gamma-n)<1$.
Therefore putting $z:=j\gamma-n$ in Lemma~\ref{lem:q_dilog}, we have the first equality.
Similarly, putting $z:=n+1-j\gamma$, we have the second equality.
\end{proof}
We prepare other two lemmas.
\begin{lem}\label{lem:gamma/2}
For a complex number $w$ with $|\Re{w}|<\Re\gamma$, we have
\begin{equation*}
  \frac{E_N\bigl(w+\gamma/2\bigr)}{E_N\bigl(w-\gamma/2+1\bigr)}
  =
  \frac{1-e^{2\pi\i w/\gamma}}{1-e^{2\pi\i w}}.
\end{equation*}
\end{lem}
\begin{proof}
By definition, we have
\begin{align*}
  &T_N(w+\gamma/2)-T_N(w-\gamma/2+1)
  \\
  =&
  \frac{1}{4}
  \int_{\Rpath}
  \frac{e^{(2w+\gamma-1)t}-e^{(2w-\gamma+1)t}}{t\sinh(t)\sinh(\gamma t)}\,dt
  \\
  =&
  \frac{1}{2}
  \int_{\Rpath}
  \frac{e^{2wt}\cosh(t)}{t\sinh(t)}\,dt
  -
  \frac{1}{2}
  \int_{\Rpath}
  \frac{e^{2wt}\cosh(\gamma t)}{t\sinh(\gamma t)}\,dt
  \\
  =&
  \frac{1}{4}\int_{\Rpath}
  \frac{e^{(2w+1)t}}{t\sinh(t)}\,dt
  +
  \frac{1}{4}\int_{\Rpath}
  \frac{e^{(2w-1)t}}{t\sinh(t)}\,dt
  \\
  &-
  \frac{1}{4}\int_{\Rpath}
  \frac{e^{(2w+\gamma)t}}{t\sinh(\gamma t)}\,dt
  -
  \frac{1}{4}\int_{\Rpath}
  \frac{e^{(2w-\gamma)t}}{t\sinh(\gamma t)}\,dt
  \\
  =&
  -\frac{1}{2}\L_1(w+1)-\frac{1}{2}\L_1(w)
  +\frac{1}{2}\L_1(w/\gamma+1)+\frac{1}{2}\L_1(w/\gamma).
\end{align*}
Taking the exponentials, we have the lemma from Lemma~\ref{lem:L_k}.
\end{proof}
\begin{lem}\label{lem:z_z+1}
For a complex number $z$ with $|\Re{z}|<\Re\gamma/2$, we have
\begin{equation*}
  \frac{E_{N}(z)}{E_{N}(z+1)}
  =
  1+e^{2\pi\i z/\gamma}.
\end{equation*}
\end{lem}
\begin{proof}
By definition, we have
\begin{align*}
  &T_{N}(z)-T_{N}(z+1)
  \\
  =&
  \frac{1}{4}\int_{\Rpath}
  \frac{e^{(2z-1)t}-e^{(2z+1)t}}{t\sinh(t)\sinh(\gamma t)}\,dt
  \\
  =&
  -\frac{1}{2}\int_{\Rpath}
  \frac{e^{2zt}}{t\sinh(\gamma t)}\,dt
  \\
  =&
  -\frac{1}{2}\int_{\gamma\Rpath}\frac{e^{2zs/\gamma}}{s\sinh(s)}\,ds
  \\
  =&
  \L_1(z/\gamma+1/2).
\end{align*}
Taking the exponentials, we get the lemma from Lemma~\ref{lem:L_k}.
\end{proof}

\section{Summation}\label{sec:summation}
In this section, we express $J_N\left(E;e^{\xi/N}\right)$ in terms of the quantum dilogarithm $T_N(z)$.
\par
We assume that $p$ and $N$ are coprime.
See Appendix for the case with $(p,N)\ne1$.
\par
If $k<N/p$, then from Corollary~\ref{cor:q_dilog} with $(j,n)=(N-l,p-1)$ and $(j,n)=(N+l,p)$, we have
\begin{align*}
  &\prod_{l=1}^{k}(1-e^{(N-l)\xi/N})(1-e^{(N+l)\xi/N})
  \\
  =&
  \prod_{l=1}^{k}(1-e^{2(N-l)\gamma\pi\i})(1-e^{2(N+l)\gamma\pi\i})
  \\
  =&
  \prod_{l=1}^{k}
  \frac{E_N\bigl((N-l-1/2)\gamma-p+1\bigr)}{E_N\bigl((N-l+1/2)\gamma-p+1\bigr)}
  \\
  &\times
  \prod_{l=1}^{k}
  \frac{E_N\bigl((N+l-1/2)\gamma-p\bigr)}{E_N\bigl((N+l+1/2)\gamma-p\bigr)}
  \\
  =&
  \frac{E_N\bigl((N-k-1/2)\gamma-p+1\bigr)}{E_N\bigl((N-1/2)\gamma-p+1\bigr)}
  \frac{E_N\bigl((N+1/2)\gamma-p\bigr)}{E_N\bigl((N+k+1/2)\gamma-p\bigr)}
  \\
  =&
  \frac{1-e^{4p\pi^2N/\xi}}{1-e^{\xi}}
  \times
  \frac{E_N\bigl((N-k-1/2)\gamma-p+1\bigr)}{E_N\bigl((N+k+1/2)\gamma-p\bigr)},
\end{align*}
where we use Lemma~\ref{lem:gamma/2} with $w=N\gamma-p$ in the last equality.
\par
Similarly, if $k$ satisfies $mN/p<k<(m+1)N/p$, then we have
\begin{align}\label{eq:qfac_E_N}
  &\prod_{l=1}^{k}(1-e^{(N-l)\xi/N})(1+e^{(N+l)\xi/N})
  \\
  =&
  \prod_{j=0}^{m-1}
  \left(
    \prod_{l=\floor{jN/p}+1}^{\floor{(j+1)N/p}}
    \frac{E_N\bigl((N-l-1/2)\gamma-p+j+1\bigr)}{E_N\bigl((N-l+1/2)\gamma-p+j+1\bigr)}
  \right.
  \notag
  \\
  &
  \left.
    \phantom{\prod_{j=0}^{m-1}\Bigl(}\times
    \prod_{l=\floor{jN/p}+1}^{\floor{(j+1)N/p}}
    \frac{E_N\bigl((N+l-1/2)\gamma-p-j\bigr)}{E_N\bigl((N+l+1/2)\gamma-p-j\bigr)}
  \right)
  \notag
  \\
  &
  \times
  \prod_{l=\floor{mN/p}+1}^{k}
  \frac{E_N\bigl((N-l-1/2)\gamma-p+m+1\bigr)}{E_N\bigl((N-l+1/2)\gamma-p+m+1\bigr)}
  \notag
  \\
  &
  \times
  \prod_{l=\floor{mN/p}+1}^{k}
  \frac{E_N\bigl((N+l-1/2)\gamma-p-m\bigr)}{E_N\bigl((N+l+1/2)\gamma-p-m\bigr)}
  \notag
  \\
  =&
  \prod_{j=0}^{m-1}
  \frac{E_N\bigl((N-\floor{(j+1)N/p}-1/2)\gamma-p+j+1\bigr)}
       {E_N\bigl((N-\floor{jN/p}-1/2)\gamma-p+j+1\bigr)}
  \notag
  \\
  &\times
  \prod_{j=0}^{m-1}
  \frac{E_N\bigl((N+\floor{jN/p}+1/2)\gamma-p-j\bigr)}
       {E_N\bigl((N+\floor{(j+1)N/p}+1/2)\gamma-p-j\bigr)}
  \notag
  \\
  &\times
  \frac{E_N\bigl((N-k-1/2)\gamma-p+m+1\bigr)}{E_N\bigl((N-\floor{mN/p}-1/2)\gamma-p+m+1\bigr)}
  \notag
  \\
  &\times
  \frac{E_N\bigl((N+\floor{mN/p}+1/2)\gamma-p-m\bigr)}{E_N\bigl((N+k+1/2)\gamma-p-m\bigr)}
  \notag
  \\
  =&
  \frac{1-e^{4pN\pi^2/\xi}}{1-e^{\xi}}
  \left(
    \prod_{j=1}^{m}
    \left(1-e^{4(p-j)N\pi^2/\xi}\right)\left(1-e^{4(p+j)N\pi^2/\xi}\right)
  \right)
  \notag
  \\
  &\times
  \frac{E_N\bigl((N-k-1/2)\gamma-p+m+1\bigr)}{E_N\bigl((N+k+1/2)\gamma-p-m\bigr)},
  \notag
\end{align}
where we use Lemma~\ref{lem:gamma/2} with $w=N\gamma-p$, and Lemma~\ref{lem:z_z+1} with $z=(N-\floor{lN/p}-1/2)\gamma-p+l$ and $z=(N+\floor{lN/p}+1/2)\gamma-p-l$ ($l=1,2,\dots,m$).
\begin{rem}
Since $\Re{\gamma}=p/N$, we have $\Re\bigl((N-\floor{lN/p}-1/2)\gamma-p+l\bigr)=-\frac{p}{N}\floor{\frac{lN}{p}}-\frac{p}{2N}+l$.
Since $lN/p$ is not an integer, we have $lN/p-1<\floor{lN/p}<lN/p$ (the equality $\floor{lN/p}=lN/p$ does not hold).
So $\left|\Re\bigl((N-\floor{lN/p}-1/2)\gamma-p+l\bigr)\right|<\Re{\gamma}/2$ and the assumption of Lemma~\ref{lem:z_z+1} holds.
\end{rem}
Therefore, from \eqref{eq:J_xi} we have
\begin{align}\label{eq:J_f_N}
  &J_N\left(E;e^{\xi/N}\right)
  \\
  =&
  \sum_{m=0}^{p-1}
  \sum_{mN/p<k<(m+1)N/p}
  e^{-k\xi}
  \prod_{l=1}^{k}\left(1-e^{(N+l)\xi/N}\right)\left(1-e^{(N-l)\xi/N}\right)
  \notag
  \\
  =&
  \frac{1-e^{4pN\pi^2/\xi}}{1-e^{\xi}}
  \sum_{m=0}^{p-1}
  \left(
    \prod_{j=1}^{m}
    \left(1-e^{4(p-j)N\pi^2/\xi}\right)\left(1-e^{4(p+j)N\pi^2/\xi}\right)
  \right.
  \notag
  \\
  &\times
  \left.\vphantom{\sum_{m=0}^{p-1}}
    \sum_{mN/p<k<(m+1)N/p}
    e^{-k\xi}
    \frac{E_N\bigl((N-k-1/2)\gamma-p+m+1\bigr)}{E_N\bigl((N+k+1/2)\gamma-p-m\bigr)}
  \right)
  \notag
  \\
  =&
  \frac{1-e^{-4pN\pi^2/\xi}}{2\sinh(u/2)}
  \notag
  \\
  &\quad\times
  \sum_{m=0}^{p-1}
  \left(
    \beta_{p,m}
    \sum_{mN/p<k<(m+1)N/p}
    \exp
    \left(
      N\times
      f_{N}\left(\frac{2k+1}{2N}-\frac{2m\pi\i}{\xi}\right)
    \right)
  \right),
  \notag
\end{align}
where we put
\begin{align}
  \beta_{p,m}
  &:=
  e^{-4mpN\pi^2/\xi}
  \prod_{j=1}^{m}
  \left(1-e^{4(p-j)N\pi^2/\xi}\right)\left(1-e^{4(p+j)N\pi^2/\xi}\right),
  \notag
  \\
  f_{N}(z)
  &:=
  \frac{1}{N}T_N\left(\frac{\xi(1-z)}{2\pi\i}-p+1\right)
  -
  \frac{1}{N}T_N\left(\frac{\xi(1+z)}{2\pi\i}-p\right)
  \label{eq:f_N}
  \\
  &\quad
  -uz+\frac{4p\pi^2}{\xi}.
  \notag
\end{align}
\begin{rem}
Since we have
\begin{equation*}
  \Re\left(\frac{\xi(1\pm z)}{2\pi\i}\right)
  =
  p(1\pm\Re{z})\pm\frac{u}{2\pi}\Im{z},
\end{equation*}
the function $f_{N}(z)$ is defined in the region
\begin{equation*}
  \left\{
    z\in\C\Bigm|-\frac{1}{2N}<\frac{u}{2p\pi}\Im{z}+\Re{z}<\frac{1}{p}+\frac{1}{2N}
  \right\}
\end{equation*}
from Lemma~\ref{lem:T_N_convergence}.
\end{rem}
\section{Approximation}\label{sec:approximation}
In the previous section, we express $J_N\left(E;e^{\xi/N}\right)$ as a sum of the function $f_N(z)$.
In this section, we approximate it by using a function that does not depend on $N$.
\par
Since $T_N(z)/N$ uniformly converges to $\Li_2\left(e^{2\pi\i z}\right)/\xi$ (Lemma~\ref{lem:T_N_L_2}), $f_{N}(z)$ uniformly converges to
\begin{equation*}
  F(z)
  :=
  \frac{1}{\xi}\Li_2\left(e^{\xi(1-z)}\right)
  -
  \frac{1}{\xi}\Li_2\left(e^{\xi(1+z)}\right)
  -uz+\frac{4p\pi^2}{\xi}
\end{equation*}
in the region
\begin{equation}\label{eq:T_N_f_N_converge}
  \left\{
    z\in\C
    \Bigm|
    \frac{\nu}{p}\le\Re{z}+\frac{u}{2p\pi}\Im{z}\le\frac{1}{p}-\frac{\nu}{p},
    \left|\Re{z}-\frac{2p\pi}{u}\Im{z}\right|\le\frac{2M\pi}{u}+1
  \right\}.
\end{equation}
By using the identity \eqref{eq:dilog}, if $z$ is in the region
\begin{equation*}
  U_0
  :=
  \left\{
    z\in\C
    \Bigm|
    0<\Re{z}+\frac{u}{2p\pi}\Im{z}<\frac{1}{p}
  \right\},
\end{equation*}
we have
\begin{align*}
  \Li_2\left(e^{\xi(1-z)}\right)
  &=
  -\Li_2\left(e^{-\xi(1-z)}\right)-\frac{\pi^2}{6}
  -\frac{1}{2}\left(\log\left(-e^{-\xi(1-z)}\right)\right)^2
  \\
  &=
  -\Li_2\left(e^{-\xi(1-z)}\right)-\frac{\pi^2}{6}
  -\frac{1}{2}(-\xi(1-z)+(2p-1)\pi\i)^2
\end{align*}
since $\Im\xi(1-z)=2p\pi-(uy+2p\pi x)$.
Similarly, we have
\begin{align*}
  \Li_2\left(e^{\xi(1+z)}\right)
  &=
  -\Li_2\left(e^{-\xi(1+z)}\right)-\frac{\pi^2}{6}
  -\frac{1}{2}\left(\log\left(-e^{-\xi(1+z)}\right)\right)^2
  \\
  &=
  -\Li_2\left(e^{-\xi(1+z)}\right)-\frac{\pi^2}{6}
  -\frac{1}{2}(-\xi(1+z)+(2p+1)\pi\i)^2
\end{align*}
since $\Im\xi(1+z)=2p\pi+(uy+2p\pi x)$.
Therefore, $F(z)$ can also be written as
\begin{equation*}
  F(z)
  =
  \frac{1}{\xi}\Li_2\left(e^{-\xi(1+z)}\right)
  -
  \frac{1}{\xi}\Li_2\left(e^{-\xi(1-z)}\right)
  +uz-2\pi\i
\end{equation*}
in $U_0$.
\par
The first derivative of $F(z)$ is
\begin{equation}\label{eq:F'}
  \frac{d}{d\,z}F(z)
  =
  \log\left(1-e^{-u-\xi z}\right)
  +
  \log\left(1-e^{-u+\xi z}\right)
  +u
  =
  \log\left(e^u+e^{-u}-e^{\xi z}-e^{-\xi z}\right)
\end{equation}
because $-\pi<\arg\left(1-e^{-u-\xi z)}\right)+\arg\left(1-e^{-u+\xi z)}\right)<\pi$ when $u$ is real from the lemma below.
Here we choose the branch of $\arg$ so that $-\pi<\arg\zeta\le\pi$ for any $\zeta\in\C$.
Note that $e^{\pm\xi z}\in\R$ if and only if $\Im(\xi z)=u\Im{z}+2p\pi\Re{z}=2k\pi$ for some $k\in\Z$, which implies that if $z\in U_0$ then $e^{\pm\xi z}\not\in\R$.
\begin{lem}
Let $a$ be a positive real number, and $w$ be a complex number with $w\not\in\R$.
Then we have $-\pi<\arg(1-aw)+\arg(1-aw^{-1})<\pi$.
\end{lem}
\begin{proof}
We may assume that $\Im{w}>0$ without loss of generality.
Then we can easily see that $-\pi<\arg(1-aw)<0$ and that $0<\arg(1-aw^{-1})<\pi$, which implies the result.
\end{proof}
\par
The second derivative of $F(z)$ equals
\begin{equation*}
  \frac{d^2}{d\,z^2}F(z)
  =
  \frac{\xi\left(e^{-\xi z}-e^{\xi z}\right)}{e^{u}+e^{-u}-e^{\xi z}-e^{-\xi z}}.
\end{equation*}
Now, define
\begin{equation}\label{eq:varphi}
  \varphi(u)
  :=
  \log
  \left(
    \cosh{u}-\frac{1}{2}-\frac{1}{2}\sqrt{(2\cosh{u}+1)(2\cosh{u}-3)}
  \right),
\end{equation}
where we take the square root as a positive multiple of $\i$, recalling that $\cosh{u}<3/2$.
Note that $\varphi(u)$ satisfies the equality
\begin{equation*}
  e^{u}+e^{-u}-e^{\varphi(u)}-e^{-\varphi(u)}=1.
\end{equation*}
\begin{lem}
If $0<u<\kappa=\arccosh(3/2)$, then $\varphi(u)$ is purely imaginary with $-\pi/3<\Im\varphi(u)<0$.
\end{lem}
\begin{proof}
First note that $e^{\varphi(u)}$ is a solution to the following quadratic equation:
\begin{equation*}
  x^2-(2\cosh{u}-1)x+1=0.
\end{equation*}
Therefore $\left|e^{\varphi(u)}\right|=1$ and we conclude that $\varphi(u)$ is purely imaginary.
Put $\theta:=\Im\varphi(u)$.
\par
Since $0<u<\kappa$, we see that $1<2\cosh{u}-1<2$.
Then since $e^{-\theta\i}$ is the other solution to the quadratic equation above, we have $2\cos\theta=2\cosh{u}-1$.
Therefore we see that $-\pi/3<\theta<0$ because the argument of $\log$ in \eqref{eq:varphi} is in the fourth quadrant.
\end{proof}
\par
As in the proof above, we put $\theta:=\Im\varphi(u)$.
We also put $\sigma_0:=\frac{(\theta+2\pi)\i}{\xi}$.
Since we have
\begin{equation*}
  \Re\sigma_0+\frac{u}{2p\pi}\Im\sigma_0
  =
  \frac{\theta+2\pi}{2p\pi}
\end{equation*}
and $0>\theta>-\pi/3$, we see that $\sigma_0\in U_0$.
\par
We have
\begin{equation*}
  \frac{d}{d\,z}F(\sigma_0)
  =
  \log\left(e^u+e^{-u}-e^{\varphi(u)}-e^{-\varphi(u)}\right)
  =0.
\end{equation*}
We also have
\begin{equation*}
  \frac{d^2}{d\,z^2}F(\sigma_0)
  =
  \xi\sqrt{(2\cosh{u}+1)(2\cosh{u}-3)}.
\end{equation*}
Therefore we conclude that $F(z)$ is of the form
\begin{equation}\label{eq:F_Taylor}
  F(z)
  =
  F(\sigma_0)
  +a_2(z-\sigma_0)^2
  +a_3(z-\sigma_0)^3
  +a_4(z-\sigma_0)^4+\cdots
\end{equation}
with $a_2:=\frac{1}{2}\xi\sqrt{(2\cosh{u}+1)(2\cosh{u}-3)}$.
\par
Now, the sum
\begin{equation}\label{eq:sum_f_N}
  \sum_{m/p<k/N<(m+1)/p}
  \exp\left(N\times f_N\left(\frac{2k+1}{2N}-\frac{2m\pi\i}{\xi}\right)\right)
\end{equation}
can be approximate by the sum
\begin{equation*}
  \sum_{m/p<k/N<(m+1)/p}
  \exp\left(N\times\Phi_m\left(\frac{2k+1}{2N}\right)\right),
\end{equation*}
where we put
\begin{equation*}
  \Phi_m(z)
  :=
  F\left(z-\frac{2m\pi\i}{\xi}\right).
\end{equation*}
Moreover, in the next section we approximate the sum \eqref{eq:sum_f_N} by the integral $N\int_{m/p}^{(m+1)/p}e^{N\Phi_m(z)}\,dz$.
\par
Note that the function $\Phi_m(z)$ is defined in the region
\begin{equation*}
  U_m
  :=
  \left\{
    z\in\C
    \Bigm|
    \frac{m}{p}<\Re{z}+\frac{u}{2p\pi}\Im{z}<\frac{m+1}{p}
  \right\}.
\end{equation*}
Put $\sigma_m:=\sigma_0+\frac{2m\pi\i}{\xi}$.
Then we see that
\begin{equation*}
  \Re{\sigma_m}+\frac{u}{2p\pi}\Im{\sigma_m}
  =
  \Re{\sigma_0}+\frac{u}{2p\pi}\Im{\sigma_0}+\frac{m}{p}
  =
  \frac{\theta+2(m+1)\pi}{2p\pi},
\end{equation*}
and so we have $\sigma_m\in U_m$.
From \eqref{eq:F_Taylor}, we conclude that $\Phi_m(z)$ is of the form
\begin{equation}\label{eq:Phi_Taylor}
  \Phi_m(z)
  =
  F(\sigma_0)
  +a_2(z-\sigma_m)^2
  +a_3(z-\sigma_m)^3
  +a_4(z-\sigma_m)^4+\cdots.
\end{equation}
\section{The Poisson summation formula}\label{sec:Poisson}
First of all, note that the function $f_N\left(z-\frac{2m\pi\i}{\xi}\right)$ uniformly converges to $\Phi_m(z)$ in the region
\begin{equation}\label{eq:Phi_converge}
  \left\{
    z\in\C
    \Bigm|
    \frac{m}{p}+\frac{\nu}{p}\le\Re{z}+\frac{u}{2p\pi}\Im{z}\le\frac{m+1}{p}-\frac{\nu}{p},
    \left|\Re{z}-\frac{2p\pi}{u}\Im{z}\right|\le\frac{2M\pi}{u}+1
  \right\}
\end{equation}
from \eqref{eq:T_N_f_N_converge}.
So we expect that the sum \eqref{eq:sum_f_N} is approximated by the integral $N\int_{m/p}^{(m+1)p}e^{N\Phi_m(z)}\,dz$ by using the Poisson summation formula \cite[Proposition~4.2]{Ohtsuki:QT2016}.
To do that we will show the following proposition, which confirms the assumption of \cite[Proposition~4.2]{Ohtsuki:QT2016}.
\begin{prop}\label{prop:Phi}
Let $m$ be an integer with $0\le m\le p-1$.
Put $b_m^{-}:=m/p+\nu/p$ and $b_m^{+}:=(m+1)/p-\nu/p$.
\par
Define
\begin{align*}
  B_m
  &:=
  \left\{\frac{k}{N}\in\R\Bigm|k\in\Z,b_m^{-}\le\frac{k}{N}\le b_m^{+}\right\},
  \\
  C_m
  &:=
  \{t\in\R\mid b_m^{-}\le t\le b_m^{+}\},
  \\
  D_m
  &:=
  \{z\in\C\mid\Re{\Phi_m(z)}<\Re{\Phi_m(\sigma_m)}\},
  \\
  E_m
  &:=
  \{z\in\C\mid b_m^{-}\le\Re{z}\le b_m^{+},|\Im{z}|\le2\Im\sigma_m\}\cap U_m
\end{align*}
Then the following hold.
\begin{enumerate}
\item
The region $E_m$ contains $\sigma_m$ and $\Phi_m(z)$ is a holomorphic function in $E_m$ of the form
\begin{equation*}
  F(\sigma_0)+a_2(z-\sigma_m)^2+a_3(z-\sigma_m)^3+a_4(z-\sigma_m)^4+\cdots
\end{equation*}
with $\Re{a_2}<0$.
\item
$D_m\cap E_m$ has two connected components.
\item
$b_m^{+}$ and $b_m^{-}$ are in different components of $D_m\cap E_m$ and moreover $\Re{\Phi_m(b_m^{\pm})}<\Re{\Phi_m(\sigma_m)}-\varepsilon_{m}$ for some $\varepsilon_{m}>0$.
\item
Both $b_m^{+}$ and $b_m^{-}$ are in a connected component of
\begin{multline*}
  \overline{R}_m
  :=
  \{x+y\i\in\C\mid b_m^{-}\le x\le b_m^{+},
  \\
  y\in[0,2\Im\sigma_m],\Re{\Phi_m(x+y\i)}<\Re{\Phi_m(\sigma_m)}+2\pi y\}
  \cap U_m.
\end{multline*}
\item
Both $b_m^{+}$ and $b_m^{-}$ are in a connected component of
\begin{multline*}
  \underline{R}_m
  :=
  \{x-y\i\in\C\mid b_m^{-}\le x\le b_m^{+},
  \\
  y\in[0,2\Im\sigma_m],\Re{\Phi_m(x-y\i)}<\Re{\Phi_m(\sigma_m)}+2\pi y\}
  \cap U_m.
\end{multline*}
\end{enumerate}
\end{prop}
\par
See Figure~\ref{fig:Poisson} for a contour plot of $\Re{\Phi_m(z)}$ with $p=3$, $m=2$, and $u=0.5$.
\begin{figure}[h]
\pic{0.5}{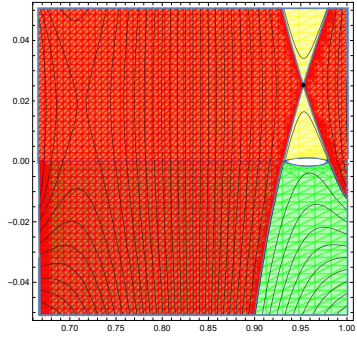}
\caption{A contour plot of $\Re{\Phi_m(z)}$ in $E_m$ by Mathematica for $p=3$, $m=2$, and $u=0.5$.
The region $\overline{R}_m$ ($\underline{R}_m$, respectively) is indicated by yellow (green, respectively).
The region $D_m$ is indicated by red, which overwrites a part of $\overline{R}_m\cup\underline{R}_m$.}
\label{fig:Poisson}
\end{figure}
\par
Before we give a proof, let us define several lines as indicated in Figure~\ref{fig:U_m}.
\begin{figure}[h]
\pic{0.4}{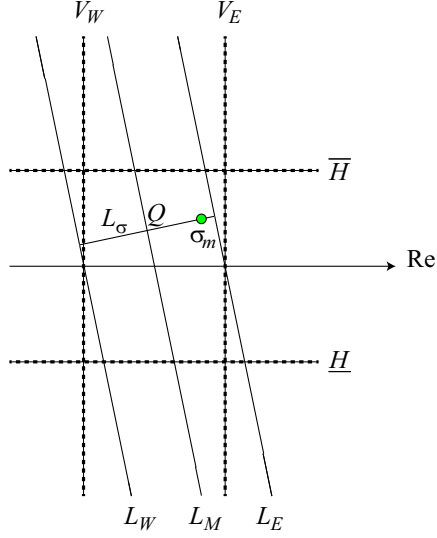}
\caption{The region $U_m$ is between $L_E$ and $L_W$.}
\label{fig:U_m}
\end{figure}
\begin{align*}
  L_{\sigma}&:\Re{z}-\frac{2p\pi}{u}\Im{z}=0,
  \\
  L_E&:\Re{z}+\frac{u}{2p\pi}\Im{z}=\frac{m+1}{p},
  \\
  L_M&:\Re{z}+\frac{u}{2p\pi}\Im{z}=\frac{2m+1}{2p},
  \\
  L_W&:\Re{z}+\frac{u}{2p\pi}\Im{z}=\frac{m}{p},
  \\
  \overline{H}&:\Im{z}=2\Im\sigma_m,
  \\
  \underline{H}&:\Im{z}=-2\Im\sigma_m,
  \\
  V_E&:\Re{z}=\frac{m+1}{p},
  \\
  V_W&:\Re{z}=\frac{m}{p}.
\end{align*}
\begin{figure}[h]
\pic{0.5}{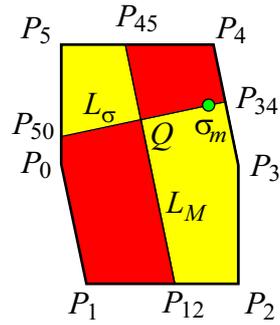}
\caption{The region $E_m$.}
\label{fig:hexagon}
\end{figure}
Note that $E_m$ is the hexagonal region surrounded by $\overline{H}$, $L_E$, $V_E$, $\underline{H}$, $L_W$, and $V_W$.
Strictly speaking, we need to push $L_E$ and $L_W$ slightly inside.
We name the vertices of its boundary as indicated in Figure~\ref{fig:hexagon}.
Their coordinates are given as:
\begin{align*}
  P_{0}&:
  \frac{m}{p},
  \\
  P_{1}&:
  \frac{m}{p}+\frac{\overline{\xi}}{p\pi}\Im\sigma_m,
  \\
  P_{2}&:
  \frac{m+1}{p}-2\Im\sigma_m\i,
  \\
  P_{3}&:
  \frac{m+1}{p},
  \\
  P_{4}&:
  \frac{m+1}{p}-\frac{\overline{\xi}}{p\pi}\Im\sigma_m,
  \\
  P_{5}&:
  \frac{m}{p}+2\Im\sigma_m\i,
\end{align*}
where $\overline{\xi}$ is the complex conjugate of $\xi$.
\par
We also put $P_{12}:=L_M\cap\underline{H}$, $P_{34}:=L_E\cap L_{\sigma}$, $P_{45}:=L_M\cap\overline{H}$, and $P_{50}:=L_W\cap L_{\sigma}$.
Their coordinates are given as follows.
\begin{align*}
  P_{12}&:
  \frac{2m+1}{2p}+\frac{\overline{\xi}}{p\pi}\Im\sigma_m,
  \\
  P_{34}&:
  \frac{2(m+1)\pi\i}{\xi},
  \\
  P_{45}&:
  \frac{2m+1}{2p}-\frac{\overline{\xi}}{p\pi}\Im\sigma_m,
  \\
  P_{50}&:
  \frac{m\overline{\xi}\i}{2p^2\pi}.
\end{align*}
We use the following lemmas in the proof of Proposition~\ref{prop:Phi} below.
\begin{lem}\label{lem:F_sigma}
We have the inequalities $0<\Re{F(0)}<\Re{F(\sigma_0)}$.
\end{lem}
\begin{lem}\label{lem:F_P_12}
We have the inequality $\Re{\Phi_m\left(P_{12}\right)}<\Re{\Phi_m(\sigma_m)}$.
\end{lem}
Proofs of the lemmas are given in Section~\ref{sec:lemmas}.
\begin{proof}[Proof of Proposition~\ref{prop:Phi}]
In the following proof, we assume that $\nu$ is sufficiently small.
We may need to modify the argument below slightly if necessary.
\par
(1).
We know that $\Phi_m(z)$ is of the form \eqref{eq:Phi_Taylor}.
Since $a_2=\frac{1}{2}\xi\i\sqrt{(2\cosh{u}+1)(3-2\cosh{u})}$ and $0<u<\arccosh(3/2)$, we see that $\Re{a_2}=-p\pi\sqrt{(2\cosh{u}+1)(3-2\cosh{u})}<0$.
So we conclude that $\Phi_m(z)$ is of this form.
\par
(2).
Writing $z=x+y\i$, we have
\begin{equation*}
  \frac{\partial}{\partial\,y}\Re{\Phi_m(x+y\i)}
  =
  -\arg\tau(x,y)
\end{equation*}
from \eqref{eq:F'}, where we put $\tau(x,y):=2\cosh(u)-2\cosh\bigl(\xi(x+y\i)\bigr)$.
Since we have
\begin{equation*}
  \Im\tau(x,y)
  =
  -2\sinh(ux-2p\pi y)\sin(uy+2p\pi x),
\end{equation*}
we see that $\Im\tau(x,y)>0$ ($\Im\tau(x,y)<0$, respectively) if and only if $ux<2p\pi y$ and $2k\pi<uy+2p\pi x<(2k+1)\pi$ for some integer $k$, or $ux>2p\pi y$ and $(2l-1)\pi<uy+2p\pi x<2l\pi$ for some integer $l$ ($ux>2p\pi y$ and $2k\pi<uy+2p\pi x<(2k+1)\pi$ for some integer $k$, or $ux<2p\pi y$ and $(2l-1)\pi<uy+2p\pi x<2l\pi$ for some integer $l$, respectively).
Since $z\in U_m$, we have $2m\pi<uy+2p\pi x<2(m+1)\pi$.
So we have
\begin{align*}
  &\frac{\partial}{\partial\,y}\Re{\Phi_m(x+y\i)}>0
  \quad\text{if and only if}
  \\
  &\quad\phantom{\text{or }}\text{$ux>2p\pi y$ and $2m\pi<uy+2p\pi x<(2m+1)\pi$}
  \\
  &\quad\text{or $ux<2p\pi y$ and $(2m+1)\pi<uy+2p\pi x<2(m+1)\pi$},
\end{align*}
and
\begin{align*}
  &\frac{\partial}{\partial\,y}\Re{\Phi_m(x+y\i)}<0
  \quad\text{if and only if}
  \\
  &\quad\phantom{\text{or }}\text{$ux<2p\pi y$ and $2m\pi<uy+2p\pi x<(2m+1)\pi$}
  \\
  &\quad\text{or $ux>2p\pi y$ and $(2m+1)\pi<uy+2p\pi x<2(m+1)\pi$}.
\end{align*}
Therefore, fixing $x$, $\Re{\Phi_m(x+y\i)}$ is monotonically increasing (decreasing, respectively) with respect to $y$ in the red region (yellow region, respectively) in Figure~\ref{fig:hexagon}.
\par
Next, we will show (i) the segment $\overline{P_{50}P_{34}}\subset L_{\sigma}$ except $\sigma_m$, (ii) the segment $\overline{P_3P_{34}}\subset L_E$, and (iii) the segment $\overline{P_{12}P_{45}}\subset L_M$ are in $D_m$.
See Figure~\ref{fig:cross}
\begin{figure}[h]
\pic{0.6}{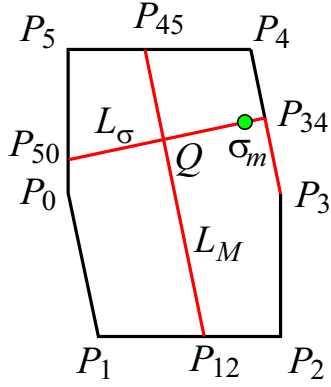}
\caption{The red segments are in $D_m$.}
\label{fig:cross}
\end{figure}
\par
(i):
Consider the segment of $L_{\sigma}$ between $L_W$ and $L_E$ that is parametrized as $\ell_{\sigma}(t):=t\sigma_m$ ($\frac{2m\pi}{2(m+1)\pi+\theta}\le t\le\frac{2(m+1)\pi}{2(m+1)\pi+\theta}$).
Then we have
\begin{align*}
  \frac{d}{d\,t}\Re\Phi_m\left(\ell_{\sigma}(t)\right)
  &=
  \Re\left(\sigma_m\log\bigl(2\cosh(u)-2\cosh(t\sigma_m\xi\bigr)\right)
  \\
  &=
  (\Re\sigma_m)
  \log\left(2\cosh(u)-2\cos\left(\bigl(\theta+2(m+1)\pi\bigr)t\right)\right).
\end{align*}
Since $2m\pi\le\bigl(2(m+1)\pi+\theta\bigr)t\le2(m+1)\pi$ and $\cosh{u}-1/2=\cosh\varphi(u)=\cos\theta$, we see that $\frac{d}{d\,t}\Re\Phi_m\left(\ell_{\sigma}(t)\right)>0$ if and only if $\frac{2m\pi-\theta}{2(m+1)\pi+\theta}<t<1$, and that $\frac{d}{d\,t}\Re\Phi_m\left(\ell_{\sigma}(t)\right)<0$ if and only if $\frac{2m\pi}{2(m+1)\pi+\theta}<t<\frac{2m\pi-\theta}{2(m+1)\pi+\theta}$ or $1<t<\frac{2(m+1)\pi}{2(m+1)\pi+\theta}$.
\par
Let $P_W$ be the point $L_{\sigma}\cap L_W$ with coordinate $\frac{2m\pi\i}{\xi}$.
Since $\Phi_m(P_W)=F(0)$ and $\Phi_m(\sigma_m)=F(\sigma_0)$, Lemma~\ref{lem:F_sigma} implies that $\Re\Phi_m\left(\ell_{\sigma}(t)\right)$ takes its maximum $\Re\Phi_m(\sigma_m)$ at $t=1$.
This shows that $L_{\sigma}\cap E_m$ is in $D_m$ except for $\sigma_m$.
\par
(ii):
Consider the segment $\overline{P_3P_4}$ that is parametrized as $\ell_E(t):=\frac{m+1}{p}-\frac{u}{2p\pi}t+t\i=\frac{m+1}{p}-\frac{\overline{\xi}}{2p\pi}t$ ($0\le t\le2\Im\sigma_m$).
We have
\begin{align*}
  &\frac{d}{d\,t}\Re{\Phi_m\left(\ell_E(t)\right)}
  \\
  =&
  -\Re
  \left(
    \frac{\overline{\xi}}{2p\pi}
    \log\left(2\cosh{u}-2\cosh\left(\xi\ell_E(t)\right)\right)
  \right)
  \\
  =&
  -
  \frac{u}{2p\pi}
  \log\left(2\cosh{u}-2\cosh\left(\frac{(m+1)u}{p}-\frac{|\xi|^2t}{2p\pi}\right)\right)
  >0,
\end{align*}
because
\begin{align*}
  &\left|
    \frac{(m+1)u}{p}-\frac{|\xi|^2t}{2p\pi}
  \right|
  \\
  \le&
  \max
  \left\{
    \frac{(m+1)u}{p},\left|\frac{(m+1)u}{p}-\frac{u\bigl(\theta+2(m+1)\pi\bigr)}{p\pi}\right|
  \right\}
  \\
  =&
  \max
  \left\{
    \frac{(m+1)u}{p},\frac{(m+1)u}{p}+\frac{u\theta}{p\pi}
  \right\}
  =
  \frac{(m+1)u}{p}
  \le u.
\end{align*}
Since the point $P_{34}$ is in $D_m$, we conclude that $\overline{P_3P_{34}}\subset D_m$.
\par
(iii):
The line $L_M$ between $\underline{H}$ and $\overline{H}$ is parametrized as $\ell_M(t):=\frac{2m+1}{2p}-\frac{u}{2p\pi}t+t\i=\frac{2m+1}{2p}-\frac{\overline{\xi}}{2p\pi}t$ ($-2\Im\sigma_m\le t\le2\Im\sigma_m$).
Now we have
\begin{align*}
  &\frac{d}{d\,t}\Re{\Phi_m\left(\ell_M(t)\right)}
  \\
  =&
  -\Re
  \left(
    \frac{\overline{\xi}}{2p\pi}
    \log
    \left(
      2\cosh(u)-2\cosh\left(\frac{(2m+1)}{2p}\xi-\frac{|\xi|^2}{2p\pi}t\right)
    \right)
  \right)
  \\
  =&
  -
  \frac{u}{2p\pi}
  \log
  \left(
    2\cosh(u)+2\cosh\left(\frac{(2m+1)u}{2p}-\frac{|\xi|^2}{2p\pi}t\right)
  \right)
  <0.
\end{align*}
Since $\ell_M(-2\Im\sigma_m)=P_{12}$, from Lemma~\ref{lem:F_P_12}, we see that $\Re{\Phi_m(P_{12})}<\Re{\Phi_m(\sigma_m)}$.
Therefore every point $z$ on $\overline{P_{12}P_{45}}$ satisfies $\Re{\Phi_m(z)}<\Re{\Phi_m(\sigma_m)}$.
\par
Now we split $E_m$ into five pieces:
\begin{align*}
  E_{m,1}
  &:=
  \{z\in E\mid b_m^{-}\le\Re{z}\le\Re{P_{45}}\},
  \\
  E_{m,2}
  &:=
  \{z\in E\mid\Re{P_{45}\le\Re{z}}\le\Re{Q}\},
  \\
  E_{m,3}
  &:=
  \{z\in E\mid\Re{Q}\le\Re{z}\le\Re{P_{12}}\},
  \\
  E_{m,4}
  &:=
  \{z\in E\mid\Re{P_{12}}\le\Re{z}\le\Re{\sigma_m}\},
  \\
  E_{m,5}
  &:=
  \{z\in E\mid\Re{\sigma_m}\le\Re{z}\le\Re{P_{34}}\},
  \\
  E_{m,6}
  &:=
  \{z\in E\mid\Re{P_{34}}\le\Re{z}\le b_m^{+}\},
\end{align*}
where $Q$ is the intersection of $L_M$ and $L_{\sigma}$.
See Figure~\ref{fig:D_m}.
\begin{figure}[h]
\pic{0.6}{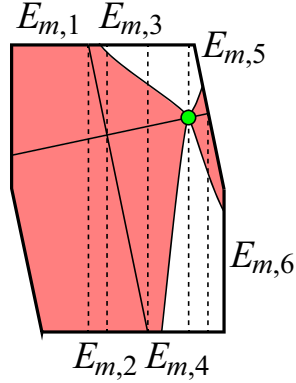}
\caption{The red region is $D_m$.}
\label{fig:D_m}
\end{figure}
\par
Note the following:
\begin{itemize}
\item
$\Re{P_1}<\Re{P_{45}}$:
This is because $\Re{P_1}-\Re{P_{45}}=-\frac{1}{2p}+2\frac{u}{p\pi}\Im\sigma_m$, which can be proved to be negative.
\item
$\Re{P_{12}}<\Re{P_4}$:
This is because $\Re{P_{12}}-\Re{P_4}=-\frac{1}{2p}+2\frac{u}{p\pi}\Im\sigma_m<0$ as above.
\item
$\Re{P_{12}}<\Re{\sigma_m}$:
This is because $\Re{P_{12}}-\Re\sigma_m=\frac{2m+1}{2p}+\frac{u}{p\pi}\Im\sigma_m-\Re\sigma_m<0$.
\item
$\Re{\sigma_m}$ can be greater than, less than, or equal to $\Re{P_4}$.
\end{itemize}
\par
In the following, we will show that any point in $\left(E_{m,1}\cup E_{m,2}\cup E_{m,3}\cup E_{m,4}\right)\cap D_m$ can be connected to a point on $L_{\sigma}$ by a segment contained in $D_m$, and that any point in $\left(E_{m,5}\cup E_{m,6}\right)\cap D_m$ can also be connected to a point on $L_{\sigma}$ by a segment contained in $D_m$.
We will also show that the vertical line through $\sigma_m$ does not intersect with $D_m$.
Then, we conclude that $D_m\cap E_m$ has two connected components $\left(E_{m,1}\cup E_{m,2}\cup E_{m,3}\cup E_{m,4}\right)\cap D_m$ and $\left(E_{m,5}\cup E_{m,6}\right)\cap D_m$ because $L_{\sigma}\setminus\{\sigma_m\}$ has two connected components.
\begin{itemize}
\item $E_{m,1}$:
Since $\Re{\Phi_m(x+y\i)}<\Re{\Phi_m(\sigma_m)}$ when $x+y\i$ is on $L_{\sigma}$ and $\Re{\Phi_m(x+y\i)}$ decreases whether $y$ increases or decreases fixing $x\in\left[b_m^{-},\Re{P_{45}}\right]$, we conclude that $\Re{\Phi_m(x+y\i)}<\Re{\Phi_m(\sigma)}$ for any $x+y\i\in E_{m,1}$.
So we can connect any point in $E_{m,1}$ to a point on $L_{\sigma}$.
\item $E_{m,2}$:
Figure~\ref{fig:E_2} indicates a graph of $\Re{\Phi_m(x+y\i)}$ for $x+y\i\in E_{m,2}$ with fixed $x$.
\begin{figure}[h]
\pic{0.4}{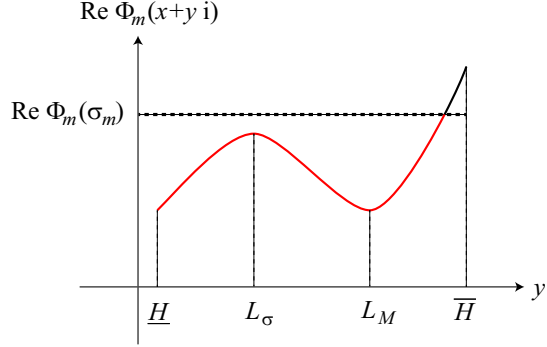}
\caption{The vertical axis is $\Re{\Phi_m(x+y\i)}$ and the horizontal axis is $y$ with fixed $x$.
The red part is included in $D_m$. Note that the local maximum is less than $\Re{\Phi_m(\sigma_m)}$.}
\label{fig:E_2}
\end{figure}
This figure shows that any point in $E_{m,2}\cap D_m$ can be connected to a point on $L_{\sigma}$ by a vertical segment in $D_m$.
\item $E_{m,3}$:
A graph of $\Re{\Phi_m(x+y\i)}$ for $x+y\i\in E_{m,3}$ with fixed $x$ looks like Figure~\ref{fig:E_2} because $\overline{P_{12}P_{45}}\subset D_m$.
\begin{figure}[h]
\pic{0.4}{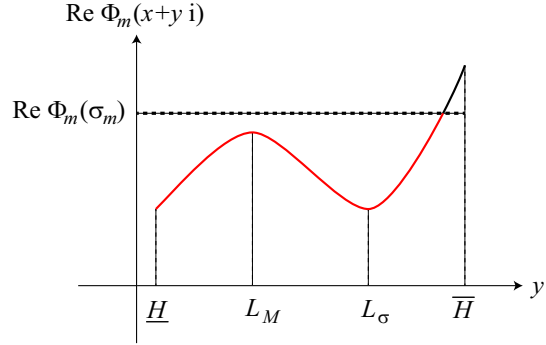}
\caption{The vertical axis is $\Re{\Phi_m(x+y\i)}$ and the horizontal axis is $y$ for fixed $x$.
The red part is included in $D_m$.}
\label{fig:E_3}
\end{figure}
Therefore the argument as before shows that any point in $E_{m,3}\cap D_m$ can be connected to a point on $L_{\sigma}$ by a vertical segment in $D_m$.
\item $E_{m,4}$:
Starting at a point on $L_{\sigma}$, whether $y$ increases or decreases, $\Re\Phi_m(x+y\i)$ increases.
Therefore any point in $E_{m,4}\cap D_m$ can be connected to a point on $L_{\sigma}$ by a vertical segment in $D_m$.
\item $E_{m,5}$:
This follows by the same reason as $E_{m,4}\cap D_m$.
\item $E_{m,6}$:
By the same argument as $E_{m,4}$, we can connect any point $z$ in $E_{m,6}\cap D_m$ to a point $z'$ in  $\overline{P_3P_{34}}$ by a vertical segment in $D_m$, and then connect $z'$ to a point in $L_{\sigma}$ by a segment in $\overline{P_3P_{34}}$.
(Precisely speaking, we need to push these segments in $E_{m,6}$.)
\end{itemize}
The fact that the vertical segment through $\sigma_m$ does not intersect with $D_m$ easily follows because $\sigma_m\not\in D_m$, and $\frac{\partial}{\partial\,y}\Re{\Phi_m(x+y\i)}$ is increasing (decreasing, respectively)  if $x+y\i$ is above $\sigma_m$ (below $\sigma_m$, respectively).
\par
See Figure~\ref{fig:D_m}.
\par
(3).
From the definition, we know that $b_m^{-}\in E_{m,1}$ and $b_m^{+}\in E_{m,6}$.
Therefore we can choose $\varepsilon_{m}$ such that $\Re{\Phi_m(b_m^{\pm})}<\Re{\Phi_m(\sigma_m)}-\varepsilon_{m}$.
\par
(4).
Since any point $z$ ($z\ne\sigma_m$) on the polygonal chain $\overline{P_0P_{50}P_{34}}P_3$ satisfies $\Re{\Phi_m(z)}<\Re{\Phi_m(\sigma_m)}$, and $\Im\sigma_m>0$, we conclude that this is in $\overline{R}_m$.
Therefore we can connect $b_m^{-}$ and $b_m^{+}$ in $\overline{R}_m$.
\par
(5).
We know that if $z$ is on the polygonal chain $\overline{P_0P_1P_{12}}$, then $\Re{\Phi_m(z)}<\Re{\Phi_m(\sigma_m)}$, which shows that $\overline{P_0P_1P_{12}}$ is in $\underline{R}_m$.
\par
We will show that the segment $\overline{P_{12}P_2}$ is also in $\underline{R}_m$.
From the proof of (2), we have $0>\frac{\partial}{\partial\,y}\Re{\Phi_m(x+y\i)}>-\pi$ if $x+y\i\in\overline{P_{12}P_2}$.
We know that if $x+y\i$ is on the polygonal chain $\overline{QP_{34}P_3}$, then $\Re{\Phi_m(x+y\i)}\le\Re{\Phi_m(\sigma_m)}$.
Since the difference of the imaginary part of $x-2\Im\sigma_m\i$ and $x+y\i$ is less than $4\Im\sigma_m$ if $x+y\i$ is on the polygonal chain $\overline{QP_{34}P_3}$, we have $\Re{\Phi_m(x-2\Im\sigma_m\i)}-\Re{\Phi_m(x+y\i)}<4\pi\Im\sigma_m$.
Therefore $\Re{\Phi_m(x-2\Im\sigma_m\i)}-\Re{\Phi_m(\sigma_m)}<2\pi\times2\Im\sigma_m$, proving that $z\in\underline{R}_m$ if $z$ is on $\overline{P_{12}P_2}$.
\par
The segment $\overline{P_2P_3}$ is also in $\underline{R}_m$.
This is because $\frac{\partial}{\partial\,y}\bigl(\Re{\Phi_m\bigl((m+1)/p+y\i\bigr)}+2\pi y\bigr)=\frac{\partial}{\partial\,y}\Re{\Phi_m\bigl((m+1)/p+y\i\bigr)}+2\pi>0$ and $P_3\in\underline{R}_m$.
\par
Now, we can connect $b_m^{-}$ and $b_m^{+}$ by the polygonal chain $\overline{P_0P_1P_2P_3}$.
\par
The proof is complete.
\end{proof}

\section{Proof of Theorem~\ref{thm:main}}\label{sec:proof}
Now we can prove Theorem~\ref{thm:main}
\begin{proof}[Proof of Theorem~\ref{thm:main}]
Since $f_N(z)$ uniformly converges to $F(z)$ in the region \eqref{eq:T_N_f_N_converge}, $f_N\left(z-\frac{2m\pi\i}{\xi}\right)$ uniformly converges to $\Phi_m(z)$ in \eqref{eq:Phi_converge}.
So we can use \cite[Proposition~4.2]{Ohtsuki:QT2016} (see also Remark~4.4 there) to conclude that
\begin{align}\label{eq:Poisson_F}
  &\frac{1}{N}\sum_{m/p+\nu/p\le k/N\le(m+1)/p-\nu/p}
  \exp\left(N\times f_{N}\left(\frac{2k+1}{2N}-\frac{2m\pi\i}{\xi}\right)\right)
  \\
  =&
  \int_{m/p+\nu/p}^{(m+1)/p-\nu/p}
  e^{N\Phi_m(z)}
  \,dz
  +
  O(e^{-N\varepsilon'_{m}})
  \notag
\end{align}
for some $\varepsilon'_{m}>0$ from Proposition~\ref{prop:Phi}.
\par
Since $\Phi_m(z)$ is of the form \eqref{eq:Phi_Taylor} in $E_m$, we can apply the saddle point method (see \cite[Proposition~3.2 and Remark~3.3]{Ohtsuki:QT2016}) to obtain
\begin{equation}\label{eq:saddle_F}
  \int_{m/p+\nu/p}^{(m+1)/p-\nu/p}e^{N\Phi_m(z)}\,dz
  =
  \frac{\sqrt{\pi}e^{N\times F(\sigma_0)}}
       {\sqrt{-\frac{1}{2}\xi\sqrt{(2\cosh{u}+1)(2\cosh{u}-3)}}\sqrt{N}}
  \bigl(1+O(N^{-1})\bigr),
\end{equation}
where we choose the sign of the outer square root so that its real part is positive (recall that we choose the sign the inner square root so that it is a positive multiple of $\i$).
From \eqref{eq:Poisson_F} and \eqref{eq:saddle_F}, we have
\begin{align}\label{eq:f_N_F}
  &\sum_{m/p+\nu/p\le k/N\le(m+1)/p-\nu/p}
  \exp\left(N\times f_{N}\left(\frac{2k+1}{2N}-\frac{2m\pi\i}{\xi}\right)\right)
  \\
  =&
  \frac{\sqrt{2\pi}e^{\pi\i/4}}{\bigl((1+2\cosh{u})(3-2\cosh{u})\bigr)^{1/4}}
  e^{N\times F(\sigma_0)}
  \times
  \sqrt{\frac{N}{\xi}}
  \bigl(1+O(N^{-1})\bigr),
  \notag
\end{align}
since $\Re{F(\sigma_0)}>0$ from Lemma~\ref{lem:F_sigma}.
\par
Now, we use the following lemma, a proof of which is given in Section~\ref{sec:lemmas}.
\begin{lem}\label{lem:f_N_end_points}
There exists $\varepsilon>0$ such that $\Re{\Phi_m\left(\frac{m+1}{p}\right)}<\Re{\Phi_m(\sigma_m)}-2\varepsilon$ for $m=0,1,2,\dots,p-1$.
Moreover there exists $\tilde{\delta}_m>0$ such that if $\frac{m}{p}\le\frac{k}{N}<\frac{m}{p}+\tilde{\delta}_m$ or $\frac{m+1}{p}-\tilde{\delta}_m<\frac{k}{N}<\frac{m+1}{p}$, then we have
\begin{equation}\label{eq:f_N_sigma}
  \Re{f_N\left(\frac{2k+1}{2N}-\frac{2m\pi\i}{\xi}\right)}
  <
  \Re{F(\sigma_0)}-\varepsilon
\end{equation}
for sufficiently large $N$.
\end{lem}
\par
If we choose $\nu$ so that $\nu/p\le\tilde{\delta}_m$, the sums
\begin{align*}
  \sum_{m/p\le k/N<m/p+\nu/p}
  \exp\left(N\times f_N\left(\frac{2k+1}{2N}-\frac{2m\pi\i}{\xi}\right)\right)
  \\
  \intertext{and}
  \sum_{(m+1)/p-\nu/p<k/N\le(m+1)/p}
  \exp\left(N\times f_N\left(\frac{2k+1}{2N}-\frac{2m\pi\i}{\xi}\right)\right)
\end{align*}
are both of order $O\left(Ne^{N(\Re{F(\sigma_0)}-\varepsilon)}\right)$ from Lemma~\ref{lem:f_N_end_points}.
Therefore we have
\begin{align*}
  &\sum_{m/p\le k/N\le(m+1)/p}
  \exp
  \left(
    N\times
    f_N\left(\frac{2k+1}{2N}-\frac{2m\pi\i}{\xi}\right)
  \right)
  \\
  =&
  \sum_{m/p+\nu/p\le k/N\le(m+1)/p-\nu/p}
  \exp
  \left(
    N\times
    f_N\left(\frac{2k+1}{2N}-\frac{2m\pi\i}{\xi}\right)
  \right)
  \\
  &+
  O\left(Ne^{N(\Re{F(\sigma_0)}-\varepsilon)}\right)
  \\
  =&
  \frac{\sqrt{2\pi}e^{\pi\i/4}}{\bigl((1+2\cosh{u})(3-2\cosh{u})\bigr)^{1/4}}
  e^{N\times F(\sigma_0)}
  \times
  \sqrt{\frac{N}{\xi}}
  \bigl(1+O(N^{-1})\bigr)
\end{align*}
where the second equality follows from \eqref{eq:f_N_F}.
\par
It follows that
\begin{align*}
  J_N\left(E;e^{\xi/N}\right)
  &=
  \frac{1}{2\sinh(u/2)}
  \left(\sum_{m=0}^{p-1}\beta_{p,m}\right)
  \times
  \frac{\sqrt{2\pi}e^{\pi\i/4}}{\bigl((1+2\cosh{u})(3-2\cosh{u})\bigr)^{1/4}}
  \\
  &\quad\times
  \sqrt{\frac{N}{\xi}}
  e^{N\times F(\sigma_0)}
  \bigl(1+O(N^{-1})\bigr)
\end{align*}
from \eqref{eq:J_f_N}.
Using \eqref{eq:J} with $N=p$ and $q=e^{4N\pi^2/\xi}$, we have
\begin{equation*}
  \sum_{m=0}^{p-1}
  \beta_{m,p}
  =
  J_{p}\left(E;e^{4N\pi^2/\xi}\right).
\end{equation*}
Therefore we have
\begin{align*}
  J_N\left(E;e^{\xi/N}\right)
  =&
  \frac{1}{2\sinh(u/2)}
  J_{p}\left(E;e^{4N\pi^2/\xi}\right)
  \times
  \frac{\sqrt{2\pi}e^{\pi\i/4}}{\bigl((1+2\cosh{u})(3-2\cosh{u})\bigr)^{1/4}}
  \\
  &\quad\times
  \sqrt{\frac{N}{\xi}}
  e^{N\times F(\sigma_0)}
  \bigl(1+O(N^{-1})\bigr).
\end{align*}
\par
Putting
\begin{align*}
  S_E(u)
  &:=
  \xi\bigl(F(\sigma_0)+2\pi\i\bigr)
  \\
  &=
  \Li_2\left(e^{-u-\varphi(u)}\right)-\Li_2\left(e^{-u+\varphi(u)}\right)
  +
  u\bigl(\varphi(u)+2\pi\i\bigr),
  \\
  T_E(u)
  &:=
  \frac{2}{\sqrt{(2\cosh{u}+1)(2\cosh{u}-3)}},
\end{align*}
we finally have
\begin{multline*}
  J_N\left(E;e^{\xi/N}\right)
  \\
  =
  \frac{\sqrt{-\pi}}{2\sinh(u/2)}T_E(u)^{1/2}
  J_{p}\left(E;e^{4N\pi^2/\xi}\right)
  \left(\frac{N}{\xi}\right)^{1/2}
  e^{\frac{N}{\xi}\times S_E(u)}
  \bigl(1+O(N^{-1})\bigr),
\end{multline*}
which proves Theorem~\ref{thm:main}.
\end{proof}
We can see that the cohomological adjoint Reidemeister torsion $T_E(u)$ equals $\pm T_E(u)$ and the Chern--Simons invariant $\CS_{u,v(u)}(\rho)$ is given by $S_E(u)-u\pi\i-\frac{1}{4}uv(u)\pmod{\pi^2\Z}$.
See for example \cite[Chapter~5]{Murakami/Yokota:2018} for calculation of the adjoint Reidemeister torsion and the Chern--Simons invariant.

\section{Quantum modularity}\label{sec:modular}
For $\eta:=\begin{pmatrix}a&b\\c&d\end{pmatrix}\in\SL(2;\Z)$ and a complex number $z$, define $\eta(z):=\dfrac{az+b}{cz+d}$ as usual.
We also define $\hbar_{\eta}(z):=\frac{2\pi\i}{z-\eta^{-1}(\infty)}=\frac{2c\pi\i}{cz+d}$.
\par
In \cite{Zagier:2010}, D.~Zagier conjectured the following.
\begin{conj}[Quantum modularity conjecture]\label{conj:modular_Zagier}
Let $K$ be a hyperbolic knot in $S^3$ and $\eta:=\begin{pmatrix}a&b\\c&d\end{pmatrix}\in\SL(2;\Z)$ with $c>0$.
Putting $X_0:=N/p$ for positive integers $N$ and $p$, the following asymptotic equivalence holds.
\begin{equation}\label{eq:modularity_Zagier}
  \frac{J_{cN+dp}\left(K;e^{2\pi\i\eta(X_0)}\right)}{J_{p}\left(K;e^{2\pi\i X_0}\right)}
  \\
  \underset{N\to\infty}{\sim}
  C_{K,\eta}
  \left(\frac{2\pi}{\hbar_{\eta}(X_0)}\right)^{3/2}
  \exp\left(\frac{\i\CV(K)}{\hbar_{\eta}(X_0)}\right),
\end{equation}
where $C_{K,\eta}$ is a complex number depending only on $\eta$ and $K$.
\end{conj}
Note that Conjecture~\ref{conj:modular_Zagier} is just a part of Zagier's original quantum modularity conjecture.
See \cite{Zagier:2010,Garoufalidis/Zagier:2021,Bettin/Drappeau:MATHA2022} for more details.
\begin{rem}
The modularity conjecture was proved by S.~Garoufalidis and D.~Zagier \cite{Garoufalidis/Zagier:2021} in the case of the figure-eight knot, and by S.~Bettin and S.~Drappeau \cite{Bettin/Drappeau:MATHA2022} for hyperbolic knots with at most seven crossings except for $7_2$.
\end{rem}
Bettin and Drappeau also proved that for the figure-eight knot $E$, $C_{E,\eta}$ is given as follows.
\begin{equation*}
  C_{E,\eta}
  =
  \frac{ce^{3\pi\i/4}}{3^{1/4}}
  \prod_{g=1}^{c}|\omega_g|^{2g/c}
  \left(\sum_{r=1}^{c}\prod_{g=1}^{r}|\omega_g|^2\right),
\end{equation*}
where $\omega_g:=1-\exp\left(2\pi\i(\frac{ag}{c}-\frac{5}{6c})\right)$.
\par
Since
\begin{equation*}
  S_E(0)
  =
  \Li_2\left(e^{\pi\i/3}\right)-\Li_2\left(e^{-\pi\i/3}\right)
  =
  \Vol\left(S^3\setminus{E}\right)\i
\end{equation*}
(see, for example, \cite[Appendix]{Milnor:BULAM382}), if $K$ is the figure-eight knot $E$ and $\eta=\begin{pmatrix}0&-1\\1&0\end{pmatrix}$, \eqref{eq:modularity_Zagier} turns out to be
\begin{equation}\label{eq:modularity_E}
  \frac{J_{N}\left(E;e^{2p\pi/N}\right)}{J_{p}\left(E;e^{2N\pi\i/p}\right)}
  \\
  \underset{N\to\infty}{\sim}
  -2\pi^{3/2}T_E(0)^{1/2}\left(\frac{N}{2p\pi\i}\right)^{3/2}
  \exp\left(\frac{NS_E(0)}{2p\pi\i}\right).
\end{equation}
Here we use the fact that $E$ is amphicheiral, that is, $E$ is equivalent to its mirror image, to conclude $J_{N}(E;q)=J_{N}(E;q^{-1})$.
Compare \eqref{eq:modularity_E} with \eqref{eq:J_p/J_1}, noting that $\xi=2p\pi\i$ when $u=0$.
\par
We can regard \eqref{eq:J_p/J_1} as a kind of quantum modularity with $\eta=\begin{pmatrix}0&-1\\1&0\end{pmatrix}$ as follows.
\par
Put $X:=\frac{2N\pi\i}{\xi}$.
Note that $\Re{X}\to\infty$ as $N\to\infty$.
We have $\eta(X)=\frac{-\xi}{2N\pi\i}$, $\exp(2\pi\i X)=e^{-4N\pi^2/\xi}$, $\exp\bigl(2\pi\i\eta(X)\bigr)=e^{-\xi/N}$, and $\hbar_{\eta}(X)=\xi/N$.
Since the figure-eight knot is amphicheiral, \eqref{eq:J_p/J_1} can be written as
\begin{equation*}
  \frac{J_N\left(E;e^{2\pi\i\eta(X)}\right)}{J_p\left(E;e^{2\pi\i X}\right)}
  \sim
  \frac{\sqrt{-\pi}}{2\sinh(u/2)}
  \left(\frac{T_E(u)}{\hbar_{\eta}(X)}\right)^{1/2}
  \exp\left(\frac{S_E(u)}{\hbar_{\eta}(X)}\right).
\end{equation*}
\par
We would like to generalize this to other elements of $\SL(2;\Z)$ and other hyperbolic knots in $S^3$.
Some computer experiments indicate the following conjecture stated in Introduction.
\begin{conj}[Quantum modularity conjecture for the colored Jones polynomial]\label{conj:QMCCJ}
Let $K\subset S^3$ be a hyperbolic knot, and $u$ a small complex number that is not a rational multiple of $\pi\i$.
For positive integers $p$ and $N$, put $\xi:=u+2p\pi\i$ and $X:=\frac{2N\pi\i}{\xi}$.
Then for any $\eta=\begin{pmatrix}a&b\\c&d\end{pmatrix}\in\SL(2;\Z)$ with $c>0$, the following asymptotic equivalence holds.
\begin{equation}\label{eq:QMCCJ}
  \frac{J_{cN+dp}\left(K;e^{2\pi\i\eta(X)}\right)}{J_{p}\left(K;e^{2\pi\i X}\right)}
  \underset{N\to\infty}{\sim}
  C_{K,\eta}(u)\frac{\sqrt{-\pi}}{2\sinh(u/2)}
  \left(\frac{T_K(u)}{\hbar_{\eta}(X)}\right)^{1/2}
  \exp\left(\frac{S_K(u)}{\hbar_{\eta}(X)}\right),
\end{equation}
where $C_{K,\eta}(u)\in\C$ does not depend on $p$.
\end{conj}
Note that $cN+dp$ comes from the denominator of $\eta(N/p)=\eta\left(X\bigm|_{u=0}\right)$.
\begin{rem}
Compare the exponent $1/2$ of $1/\hbar_{\eta}(X)=\frac{cX+d}{2c\pi\i}$ in \eqref{eq:QMCCJ} with $3/2$ in \eqref{eq:modularity_Zagier}.
Our modularity would have weight $1/2$ rather than $3/2$.
\end{rem}
\begin{rem}
Since $(-\eta)(X)=\eta(X)$, we may assume that $c\ge0$.
\par
If $c=0$, then $\eta=\pm\begin{pmatrix}1&k\\0&1\end{pmatrix}$ for some integer $k$.
Since $\eta(X)=X+k$, we have $\exp\bigl(2\pi\i\eta(X)\bigr)=\exp(2\pi\i X)$ and so $J_{p}\left(E;e^{2\pi\i\eta(X)}\right)=J_{p}\left(E;e^{2\pi\i X}\right)$.
\end{rem}
\begin{rem}
When $p=1$ and $\eta=\begin{pmatrix}0&-1\\1&0\end{pmatrix}$, \eqref{eq:QMCCJ} becomes
\begin{multline*}
  J_{N}\left(K;e^{-(u+2\pi\i)/N}\right)
  \\
  \underset{N\to\infty}{\sim}
  C_{K,\eta}(u)\frac{\sqrt{-\pi}}{2\sinh(u/2)}
  \left(T_K(u)\right)^{1/2}\left(\frac{N}{u+2\pi\i}\right)^{1/2}
  \exp\left(\frac{N\times S_K(u)}{u+2\pi\i}\right),
\end{multline*}
which coincides with \cite[Conjecture~1.6]{Murakami:JTOP2013} with $C_{K,\eta}(u)=1$.
See also \cite{Dimofte/Gukov:Columbia,Gukov/Murakami:FIC2008}.
Strictly speaking, we need to take the mirror image $\overline{K}$ of $K$ because $J_N\left(\overline{K};e^{-(u+2\pi\i)/N}\right)=J_N\left(K;e^{(u+2\pi\i)/N}\right)$.
\end{rem}

\section{Lemmas}\label{sec:lemmas}
In this section we prove lemmas that we use.
\begin{proof}[Proof of Lemma~\ref{lem:T_N_convergence}]
Recall that $\xi=u+2p\pi\i$ and $\gamma=\frac{\xi}{2N\pi\i}$.
\par
Since $\Re{\gamma}=p/N>0$, $\sinh(\gamma x)\underset{N\to\infty}{\sim}\frac{e^{\xi x}}{2}$ and $\sinh(\gamma x)\underset{N\to-\infty}{\sim}\frac{-e^{-\xi x}}{2}$.
So we have
\begin{align*}
  \frac{e^{(2z-1)x}}{x\sinh(x)\sinh(\gamma x)}
  &\underset{N\to\infty}{\sim}
  \frac{1}{2x}e^{(2z-\gamma-2)x},
  \\
  \intertext{and}
  \frac{e^{(2z-1)x}}{x\sinh(x)\sinh(\gamma x)}
  &\underset{N\to-\infty}{\sim}
  \frac{-1}{2x}e^{(2z+\gamma)x}.
\end{align*}
Therefore if $-\Re\gamma/2<\Re{z}<1+\Re\gamma/2$, then the integral converges, completing the lemma.
\end{proof}
The following proof is almost the same as \cite[Proposition~2.8]{Murakami/Tran:2021}.
See also \cite[Proposition~A.1]{Ohtsuki:QT2016}.
\begin{proof}[Proof of Lemma~\ref{lem:T_N_L_2}]
We will show that $T_N(z)=\frac{N}{\xi}\L_2(z)+O(1/N)$.
\par
Recalling that $\xi=2N\pi\gamma\i$, we have
\begin{align*}
  \left|T_N(z)-\frac{N}{\xi}\L_2(z)\right|
  &=
  \frac{1}{4}\int_{\Rpath}
  \left|
    \frac{e^{(2z-1)x}}{\gamma x^2\sinh(x)}\left(\frac{\gamma x}{\sinh(\gamma x)}-1\right)
  \right|\,dx
  \\
  &\le
  \frac{N\pi}{2|\xi|}\int_{\Rpath}
  \left|
    \frac{e^{(2z-1)x}}{x^2\sinh(x)}\left(\frac{\gamma x}{\sinh(\gamma x)}-1\right)
  \right|\,dx.
\end{align*}
Since the Taylor expansion of $\frac{\sinh(y)}{y}$ around $y=0$ is $1+\frac{y^2}{6}+\cdots$, we have $\frac{y}{\sinh(y)}=1-\frac{y^2}{6}+o(y^2)$ as $y\to0$.
Therefore, we have $\left|\frac{\gamma x}{\sinh(\gamma x)}-1\right|\le\frac{c|x|^2}{N^2}$ for some constant $c>0$ and so
\begin{equation*}
  \left|T_N(z)-\frac{N}{\xi}\L_2(z)\right|
  <
  \frac{c'}{N}\int_{\Rpath}
  \left|\frac{e^{(2z-1)x}}{\sinh(x)}\right|\,dx,
\end{equation*}
where we put $c':=\frac{c\pi}{2|\xi|}$.
\par
We put
\begin{align*}
  I_{+}
  &:=
  \int_{1}^{\infty}\left|\frac{e^{(2z-1)x}}{\sinh(x)}\right|\,dx,
  \\
  I_{-}
  &:=
  \int_{-\infty}^{-1}\left|\frac{e^{(2z-1)x}}{\sinh(x)}\right|\,dx,
  \\
  I_{0}
  &:=
  \int_{|x|=1,\Im{x}\ge0}\left|\frac{e^{(2z-1)x}}{\sinh(x)}\right|\,dx.
\end{align*}
We have
\begin{align*}
  I_{+}
  &\le
  \int_{1}^{\infty}\frac{2e^{2x\Re{z}-x}}{e^{x}-e^{-x}}\,dx
  =
  \int_{1}^{\infty}\frac{2e^{2x(\Re{z}-1)}}{1-e^{-2x}}\,dx
  \le
  \frac{2}{1-e^{-2}}\int_{1}^{\infty}e^{-2\nu x}\,dx
  \\
  &=
  \frac{e^{-2\nu}}{\nu(1-e^{-2})},
\end{align*}
where we use the assumption $\Re{z}\le1-\nu$.
\par
Similarly, we have
\begin{align*}
  I_{-}
  &\le
  \int_{-\infty}^{-1}\frac{2e^{2x\Re{z}-x}}{e^{-x}-e^{x}}\,dx
  =
  \int_{-\infty}^{-1}\frac{2e^{2x\Re{z}}}{1-e^{2x}}\,dx
  \le
  \frac{2}{1-e^{-2}}\int_{-\infty}^{-1}e^{2\nu x}\,dx
  \\
  &=
  \frac{e^{-2\nu}}{\nu(1-e^{-2})},
\end{align*}
where we use the assumption $\Re{z}\ge\nu$.
\par
Putting $x=e^{t\i}$ ($0\le t\le\pi$) and $L:=\max_{|x|=1,\Im{x}\ge0}|\sinh(x)|$, we have
\begin{equation*}
  I_{0}
  =
  \int_{0}^{\pi}
  \left|\frac{e^{(2z-1)e^{t\i}}}{\sinh\left(e^{t\i}\right)}\right|
  \times\left|\i e^{t\i}\right|\,dt
  \le
  \frac{1}{L}
  \int_{0}^{\pi}
  e^{(2\Re{z}-1)\cos{t}-2\Im{z}\sin{t}}\,dt,
\end{equation*}
which is bounded from the above because both $\Re{z}$ and $\Im{z}$ are bounded.
\par
Therefore, we see that $I_{+}+I_{-}+I_{0}$ is bounded from above, which implies that $\left|T_N(z)-\frac{N}{\xi}\L_2(z)\right|=O(1/N)$.
\end{proof}
\begin{proof}[Proof of Lemma~\ref{lem:F_sigma}]
Since $\Re{F(0)}$ coincides with $\Re{\Phi(w_0)}$ in \cite{Murakami:JTOP2013} (see \cite[Remark~1.6]{Murakami/Tran:2021}), we have $\Re{F(0)}>0$ from \cite[Lemma~3.5]{Murakami:JTOP2013}.
\par
Next, we will show that $\xi\bigl(F(\sigma_0)-F(0)\bigr)$ is purely imaginary with positive imaginary part.
Then we conclude that $\Re\bigl(F(\sigma_0)-F(0)\bigr)>0$, since $\xi$ is in the first quadrant.
\par
Since $\varphi(u)$ is purely imaginary, we have $\overline{\Li_2\left(e^{-u-\varphi(u)}\right)}=\Li_2\left(e^{-u+\varphi(u)}\right)$.
So we see that $\xi\bigl(F(\sigma_0)-F(0)\bigr)=\Li_2\left(e^{-u-\varphi(u)}\right)-\Li_2\left(e^{-u+\varphi(u)}\right)+u(\theta+2\pi)\i$ is purely imaginary with imaginary part $2\Im\Li_2\left(e^{-u-\varphi(u)}\right)+u(\theta+2\pi)$, which coincides with $\Im\bigl(\xi\Phi(w_0)\bigr)+2u\pi>0$ in \cite[P.~214]{Murakami:JTOP2013}.
\par
This proves the lemma.
\end{proof}
\begin{proof}[Proof of Lemma~\ref{lem:F_P_12}]
We have
\begin{align*}
  &\xi\left(\Phi_m(P_{12})-\Phi_m(\sigma_m)\right)
  \\
  =&
  \Li_2\left(-e^{-u-\frac{u((6m+5)\pi+2\theta)}{2p\pi}}\right)
  -
  \Li_2\left(-e^{-u+\frac{u((6m+5)\pi+2\theta)}{2p\pi}}\right)
  \\
  &-
  \Li_2\left(e^{-u-\varphi(u)}\right)
  +
  \Li_2\left(e^{-u+\varphi(u)}\right)
  \\
  &+
  \frac{(2m+1)u\xi}{2p}+\frac{u^2(2(m+1)\pi+\theta)}{p\pi}
  -
  u\bigl(2(m+1)\pi+\theta\bigr)\i.
\end{align*}
Its real part is
\begin{equation*}
  \Li_2\left(-e^{-u-q_m(u)}\right)
  -
  \Li_2\left(-e^{-u+q_m(u)}\right)
  +
  uq_m(u),
\end{equation*}
where we put $q_m(u):=\frac{u((6m+5)\pi+2\theta)}{2p\pi}$, and its imaginary part is
\begin{equation*}
  -2\Im\Li_2\left(e^{-u-\varphi(u)}\right)
  -
  u(\pi+\theta).
\end{equation*}
Then we have
\begin{align*}
  &\frac{|\xi|^2}{u}\Re\left(F(P_{12})-F(\sigma_m)\right)
  \\
  =&
  \Re\left(\xi\bigl(F(P_{12})-F(\sigma_m)\bigr)\right)
  +
  \frac{2p\pi}{u}\Im\left(\xi\bigl(F(P_{12})-F(\sigma_m)\bigr)\right)
  \\
  =&
  \Li_2\left(-e^{-u-q_m(u)}\right)
  -
  \Li_2\left(-e^{-u+q_m(u)}\right)
  +uq_m(u)
  \\
  &-
  \frac{2p\pi}{u}
  \left(2\Im\Li_2\left(e^{-u-\varphi(u)}\right)+u(\pi+\theta)\right).
\end{align*}
By using the inequality $2\Im\Li_2\left(e^{-u-\varphi(u)}\right)+u\theta>0$ in \cite[\S~7]{Murakami:JTOP2013}, this is less than $c_{p,m}(u)$, where we put
\begin{equation*}
  c_{p,m}(u)
  :=
  \Li_2\left(-e^{-u-q_m(u)}\right)
  -
  \Li_2\left(-e^{-u+q_m(u)}\right)
  +uq_m(u)-2p\pi^2.
\end{equation*}
Now we have
\begin{equation*}
  \frac{d}{d\,u}c_{p,m}(u)
  =
  q'_m(u)\log\bigl(2\cosh{u}+2\cosh{q_m(u)}\bigr)
  +
  \log\left(\frac{e^{q_m(u)}+e^{-u}}{1+e^{-u+q_m(u)}}\right),
\end{equation*}
which can be easily seen to be positive.
Since $u<\kappa$, it suffices to prove $c_{p,m}(\kappa)<0$.
Since $\varphi(\kappa)=0$, we have
\begin{equation*}
  c_{p,m}(\kappa)
  =
  \Li_2\left(-e^{-\kappa(1+\frac{6m+5}{2p})}\right)
  -
  \Li_2\left(-e^{-\kappa(1-\frac{6m+5}{2p})}\right)
  +\frac{(6m+5)\kappa^2}{2p}-2p\pi^2,
\end{equation*}
which is increasing with respect to $m$, fixing $p$.
We will prove that $c_{p,p-1}(\kappa)<0$.
\par
We calculate
\begin{equation*}
  c_{p,p-1}(\kappa)
  =
  \Li_2\left(-e^{\kappa(\frac{1}{2p}-4)}\right)
  -
  \Li_2\left(-e^{\kappa(-\frac{1}{2p}+2)}\right)
  +\left(3-\frac{1}{2p}\right)\kappa^2-2p\pi^2.
\end{equation*}
The derivative of $c_{p,p-1}(\kappa)$ with respect to $p$ equals
\begin{equation*}
  \frac{\kappa}{2p^2}\log\left(3+2\cosh\left(\kappa\left(3-\frac{1}{2p}\right)\right)\right)
  -2\pi^2,
\end{equation*}
which is less than $-2\pi^2+\log(6+2\cosh(3\kappa)=-18.274\ldots<0$.
It follows that $c_{p,p-1}(\kappa)<c_{1,0}(\kappa)=-14.9942\ldots<0$.
\par
This shows that $\Re\bigl(F(P_{12})-F(\sigma_m)\bigr)<0$, proving the lemma.
\end{proof}
Before proving Lemma~\ref{lem:f_N_end_points}, we prepare the following lemma.
\par
\begin{lem}\label{lem:J_m/p}
Put $g(x):=4\sinh\left(\frac{\xi}{2}(1+x)\right)\sinh\left(\frac{\xi}{2}(1-x)\right)$.
For an integer $0\le m\le p$, there exists $\delta_m>0$ such that $|g(l/N)|<1$ if $\frac{m}{p}-\delta_m<\frac{l}{N}<\frac{m}{p}+\delta_m$.
\end{lem}
\begin{proof}
For an integer $0\le m\le p$, we can easily see that
\begin{equation*}
  g(m/p)
  =
  2\bigl(\cosh{u}-\cosh(mu/p)\bigr).
\end{equation*}
So we conclude that $g(m/p)$ is monotonically decreasing with respect to $m$.
Therefore we have $0=g(1)\le g(m/p)\le g(0)=2\bigl(\cosh(u)-1\bigr)<2\cosh(\kappa)-2=1$.
So we have $0\le g(m/p)<1$.
\par
Therefore, there exists $\delta_m>0$ such that $|g(x)|<1$ if $|x-m/p|<\delta_m$, completing the proof.
\end{proof}
\begin{proof}[Proof of Lemma~\ref{lem:f_N_end_points}]

From (2)-(ii) of the proof of Proposition~\ref{prop:Phi}, we know that $\frac{m+1}{p}\in D_{m}$, that is, $\Re{\Phi_m\left(\frac{m+1}{p}\right)}<\Re{\Phi_m(\sigma_{m})}$.
Therefore there exists $\varepsilon>0$ such that $\Re{\Phi_m\left(\frac{m+1}{p}\right)}<\Re{\Phi_m(\sigma_m)}-2\varepsilon$ for $m=0,1,2,\dots,p-1$.
\par
Next, we show that there exists $\tilde{\delta}_m>0$ such that if $\frac{m+1}{p}-\tilde{\delta}_m<\frac{k}{N}<\frac{m+1}{p}$, then \eqref{eq:f_N_sigma} holds.
\par
We can choose $\delta'_m>0$ so that $\Re{\Phi_m\left(\frac{k}{N}\right)}<\Re{\Phi_m\left(\frac{m+1}{p}\right)}+\varepsilon$ if $(m+1)/p-\delta'_m<k/N<(m+1)/p$.
So we have $\Re{\Phi_m\left(\frac{k}{N}\right)}<\Re{\Phi_m(\sigma_m)}-\varepsilon$.
Now recall that $f_N(z)$ converges to $F(z)$ in the region \eqref{eq:T_N_f_N_converge}.
Since we have
\begin{align*}
  &\Re\left(\frac{2k+1}{2N}-\frac{2(m-1)\pi\i}{\xi}\right)
  +
  \frac{u}{2p\pi}
  \Im\left(\frac{2k+1}{2N}-\frac{2m\pi\i}{\xi}\right)
  \\
  =&
  \frac{2k+1}{2N}-\frac{m}{p},
  \\
  &\Re\left(\frac{2k+1}{2N}-\frac{2m\pi\i}{\xi}\right)
  -
  \frac{2p\pi}{u}
  \Im\left(\frac{2k+1}{2N}-\frac{2m\pi\i}{\xi}\right)
  \\
  =&
  \frac{2k+1}{2N},
\end{align*}
if $\nu/p+m/p-1/(2N)\le k/N\le(m+1)/p-\nu/p-1/(2N)$ and $k/N\le2M\pi/u+1-1/(2N)$, then $f_N\left(\frac{2k+1}{2N}-\frac{2m\pi\i}{\xi}\right)$ converges to
\begin{equation*}
  F\left(\frac{k}{N}-\frac{2m\pi\i}{\xi}\right)
  =
  \Phi_m\left(\frac{k}{N}\right)
\end{equation*}
as $N\to\infty$.
Therefore we see
\begin{align*}
  \Re{f_N\left(\frac{2k+1}{2N}-\frac{2m\pi\i}{\xi}\right)}
  &<
  \Re\Phi_m(\sigma_m)-\varepsilon
  \\
  &=
  \Re{F(\sigma_0)}-\varepsilon
\end{align*}
if we choose $\nu$ small enough so that $\delta'_m>\frac{\nu}{p}+\frac{1}{2N}$ (and $N$ is large enough).
Note that so far $k$ should satisfy the inequalities
\begin{equation}\label{eq:right_end}
  \frac{m+1}{p}-\delta'_m<\frac{k}{N}\le\frac{m+1}{p}-\frac{\nu}{p}-\frac{1}{2N}.
\end{equation}
\par
On the other hand, putting $h_N(k):=\prod_{l=1}^{k}g\left(\frac{l}{N}\right)$, we have
\begin{equation}\label{eq:h_N}
  \left|h_N(k)\right|>\left|h_N(k')\right|
\end{equation}
if $\frac{m}{p}-\delta_m<\frac{k}{N}<\frac{k'}{N}<\frac{m}{p}+\delta_m$ from Lemma~\ref{lem:J_m/p}.
Note that if $\frac{m}{p}\le\frac{k}{N}<\frac{m+1}{p}$, we have
\begin{equation}\label{eq:h_N_f_N}
  h_N(k)
  =
  \frac{1-e^{-4pN\pi^2/\xi}}{2\sinh(u/2)}\beta_{p,m}
  \exp\left(N\times f_N\left(\frac{2k+1}{2N}-\frac{2m\pi\i}{\xi}\right)\right)
\end{equation}
from \eqref{eq:J_f_N}.
From \eqref{eq:h_N} and \eqref{eq:h_N_f_N}, if $\frac{m+1}{p}-\delta_{m+1}<\frac{k}{N}<\frac{k'}{N}<\frac{m+1}{p}$, then we have
\begin{align*}
  \Re{f_N\left(\frac{2k+1}{2N}-\frac{2m\pi\i}{\xi}\right)}
  &=
  \frac{1}{N}
  \log\left|\frac{2\sinh(\xi/2)}{1-e^{-4pN\pi^2/\xi}}\beta_{p,m}^{-1}h_N(k)\right|
  \\
  &>
  \frac{1}{N}
  \log\left|\frac{2\sinh(\xi/2)}{1-e^{-4pN\pi^2/\xi}}\beta_{p,m}^{-1}h_N(k')\right|
  \\
  &=
  \Re{f_N\left(\frac{2k'+1}{2N}-\frac{2m\pi\i}{\xi}\right)},
\end{align*}
which means that $\Re{f_N\left(\frac{2k+1}{2N}-\frac{2m\pi\i}{\xi}\right)}$ is monotonically decreasing with respect to $k$ if $\frac{m+1}{p}-\delta_{m+1}<\frac{k}{N}<\frac{m+1}{p}$.
Combined with \eqref{eq:right_end}, we conclude that \eqref{eq:f_N_sigma} holds if $\frac{m+1}{p}-\delta'_m<\frac{k}{N}<\frac{m+1}{p}$, choosing $\delta'_m$ less than $\delta_{m+1}$ if necessary.
\par
Now, we show that for $m=1,2,\dots,p-1$, \eqref{eq:f_N_sigma} holds if $\frac{m}{p}\le\frac{k}{N}<\frac{m}{p}+\delta_m$.
\par
From \eqref{eq:h_N} and \eqref{eq:h_N_f_N}, if $\frac{m}{p}-\delta'_m<\frac{k'}{N}<\frac{m}{p}\le\frac{k}{N}<\frac{m}{p}+\delta_m$, we have
\begin{align*}
  \Re{f_N\left(\frac{2k+1}{2N}-\frac{2m\pi\i}{\xi}\right)}
  &=
  \frac{1}{N}
  \log\left|\frac{2\sinh(\xi/2)}{1-e^{-4pN\pi^2/\xi}}\beta_{p,m}^{-1}h_N(k)\right|
  \\
  &<
  \frac{1}{N}
  \log\left|\frac{2\sinh(\xi/2)}{1-e^{-4pN\pi^2/\xi}}\beta_{p,m}^{-1}h_N(k')\right|
  \\
  &<
  \frac{1}{N}
  \log\left|\frac{2\sinh(\xi/2)}{1-e^{-4pN\pi^2/\xi}}\beta_{p,m-1}^{-1}h_N(k')\right|
  \\
  &=
  \Re{f_N\left(\frac{2k'+1}{2N}-\frac{2(m-1)\pi\i}{\xi}\right)},
\end{align*}
which is less than $\Re{F(\sigma_0)}-\varepsilon$ from the argument above.
Here the second inequality follows since
\begin{align*}
  \left|\frac{\beta_{p,m}}{\beta_{p,m-1}}\right|
  &=
  2\Bigl|\cosh(4pN\pi^2/\xi)-\cosh(4mN\pi^2/\xi)\bigr|
  \\
  &\underset{N\to\infty}{\sim}
  \frac{1}{2}\exp\left(\frac{4pu\pi^2\times N}{|\xi|^2}\right).
\end{align*}
So \eqref{eq:f_N_sigma} holds.
\par
Finally, we consider the case where $m=0$.
Since $h_N(0)=\beta_{p,0}=1$, we have
\begin{align*}
  \Re{f_N\left(\frac{1}{2N}\right)}
  &=
  \frac{1}{N}\log\left|\frac{2\sinh(\xi/2)}{1-e^{-4pN\pi^2/\xi}}\right|
  \le
  \frac{1}{N}
  \log\left|\frac{2\sinh(\xi/2)}{1+\left|e^{-4pN\pi^2/\xi}\right|}\right|
  \\
  &=
  \frac{1}{N}
  \log\left|\frac{2\sinh(\xi/2)}{1+e^{-4puN\pi^2/|\xi|^2}}\right|
  \to0
  \quad\text{($N\to\infty$)}.
\end{align*}
Since $\Re{F(\sigma_0)}>0$ from Lemma~\ref{lem:F_sigma}, \eqref{eq:f_N_sigma} holds if $k/N<\delta_0$ and $N$ is sufficiently large.
\par
As a result, if we put $\tilde{\delta}_m:=\min\{\delta'_m,\delta_m\}$, \eqref{eq:f_N_sigma} holds.
\end{proof}

\appendix
\section{The case where $(p,N)\ne1$}
In this appendix, we will calculate $\prod_{l=1}^{k}\left(1-e^{(N-l)\xi/N}\right)\left(1-e^{(N+l)\xi/N}\right)$ assuming $(p,N)=c>1$.
Put $N':=N/c\in\N$ and $p':=p/c\in\N$.
\par
Note that $jN/p$ ($1\le j\le N-1$, $j\in\N$) is an integer if and only $j$ is a multiple of $p'$.
\par
If $k<N'$, then we can choose an integer $m<p'$ so that $mN/p<k<(m+1)N/p$ because $N/p,2N/p,\dots,(p'-1)N/p$ are not integers.
Therefore from \eqref{eq:qfac_E_N}, we have
\begin{align*}
  &\prod_{l=1}^{k}(1-e^{(N-l)\xi/N})(1+e^{(N+l)\xi/N})
  \\
  =&
  \frac{1-e^{4pN\pi^2/\xi}}{1-e^{\xi}}
  \left(
    \prod_{j=1}^{m}
    \left(1-e^{4(p-j)N\pi^2/\xi}\right)\left(1-e^{4(p+j)N\pi^2/\xi}\right)
  \right)
  \\
  &\times
  \frac{E_N\bigl((N-k-1/2)\gamma-p+m+1\bigr)}{E_N\bigl((N+k+1/2)\gamma-p-m\bigr)}.
\end{align*}
\par
If $k=N'$, we have
\begin{align*}
  &\prod_{l=1}^{N'}\left(1-e^{(N-l)\xi/N}\right)\left(1-e^{(N+l)\xi/N}\right)
  \\
  =&
  \left(\prod_{l=1}^{N'-1}\left(1-e^{(N-l)\xi/N}\right)\left(1-e^{(N+l)\xi/N}\right)\right)
  \left(1-e^{(N-N')\xi/N}\right)\left(1-e^{(N+N')\xi/N}\right)
  \\
  =&
  \left(1-e^{(c-1)\xi/c}\right)\left(1-e^{(c+1)\xi/c}\right)
  \\
  &\times
  \frac{1-e^{4pN\pi^2/\xi}}{1-e^{\xi}}
  \left(
    \prod_{j=1}^{p'-1}
    \left(1-e^{4(p-j)N\pi^2/\xi}\right)
    \left(1-e^{4(p+j)N\pi^2/\xi}\right)
  \right)
  \\
  &\times
  \frac{E_N\bigl((N-p'+1/2)\gamma-p+p'\bigr)}{E_N\bigl((N+p'-1/2)\gamma-p-p'+1\bigr)},
\end{align*}
since $(p'-1)N/p<N'-1<p'N/p$.
\par
If $k$ is an integer with $nN'\le k<(n+1)N'$, writing $l$ ($0\le l\le k$) as $l=aN'+b$ with $0\le a\le n$ and $0\le b\le N'-1$, we have
\begin{align*}
  &\prod_{l=1}^{k}\left(1-e^{(N-l)\xi/N}\right)\left(1-e^{(N+l)\xi/N}\right)
  \\
  =&
  \prod_{b=1}^{N'-1}
  \left(1-e^{(N-b)\xi/N}\right)\left(1-e^{(N+b)\xi/N}\right)
  \\
  &\times
  \prod_{a=1}^{n-1}\prod_{b=0}^{N'-1}
  \left(1-e^{(N-aN'-b)\xi/N}\right)\left(1-e^{(N+aN'+b)\xi/N}\right)
  \\
  &\times
  \prod_{b=0}^{k-nN'}
  \left(1-e^{(N-nN'-b)\xi/N}\right)\left(1-e^{(N+nN'+b)\xi/N}\right)
  \\
  =&
  \prod_{a=1}^{n-1}
  \left(1-e^{(c-a)\xi/c}\right)\left(1-e^{(c+a)\xi/c}\right)
  \times
  \prod_{a=0}^{n-1}\prod_{b=1}^{N'-1}P_{a,b}
  \\
  &\times
  \left(1-e^{(c-n)\xi/c}\right)\left(1-e^{(c+n)\xi/c}\right)
  \prod_{b=1}^{k-nN'}Q_b,
\end{align*}
where we put
\begin{align*}
  P_{a,b}
  &:=
  \left(1-e^{(N-aN'-b)\xi/N}\right)\left(1-e^{(N+aN'+b)\xi/N}\right)
  \\
  &=
  \left(1-e^{2(N-aN'-b)\pi\i\gamma}\right)\left(1-e^{2(N+aN'+b)\pi\i\gamma}\right),
  \\
  Q_b&:=
  \left(1-e^{(N-nN'-b)\xi/N}\right)\left(1-e^{(N+nN'+b)\xi/N}\right)
  \\
  &=
  \left(1-e^{2(N-nN'-b)\pi\i\gamma}\right)\left(1-e^{2(N+nN'+b)\pi\i\gamma}\right).
\end{align*}
If we choose $i$ ($0\le i\le p'-1$) with $iN'/p'<b<(i+1)N'/p'$, then we have $(p'-ap'-i-1)N/p<N-aN'-b<(p'-ap'-i)N/p$ and  $(p'+ap'+i)N/p<N+aN'+b<(p'+ap'+i+1)N/p$.
So from Corollary~\ref{cor:q_dilog}, we have
\begin{align*}
  &\prod_{b=1}^{N'-1}P_{a,b}
  =
  \prod_{i=0}^{p'-1}\left(\prod_{iN'/p'<b<(i+1)N'/p'}P_{a,b}\right)
  \\
  =&
  \prod_{i=0}^{p'-1}
  \left(
    \prod_{iN'/p'<b<(i+1)N'/p'}
    \frac{E_N\bigl((N-aN'-b-1/2)\gamma-p+ap'+i+1\bigr)}
         {E_N\bigl((N-aN'-b+1/2)\gamma-p+ap'+i+1\bigr)}
  \right.
  \\
  &\phantom{\prod_{i=0}^{p'-1}}\quad\times
  \left.
    \prod_{iN'/p'<b<(i+1)N'/p'}
    \frac{E_N\bigl((N+aN'+b-1/2)\gamma-p-ap'-i\bigr)}
         {E_N\bigl((N+aN'+b+1/2)\gamma-p-ap'-i\bigr)}
  \right)
  \\
  =&
  \prod_{i=0}^{p'-2}
  \left(
  \frac{E_N\bigl((N-aN'-\floor{(i+1)N'/p'}-1/2)\gamma-p+ap'+i+1\bigr)}
       {E_N\bigl((N-aN'-\floor{iN'/p'}-1/2)\gamma-p+ap'+i+1\bigr)}
  \right.
  \\
  &\phantom{\prod_{i=0}^{p'-2}}\quad\times
  \left.
  \frac{E_N\bigl((N+aN'+\floor{iN'/p'}+1/2)\gamma-p-ap'-i\bigr)}
       {E_N\bigl((N+aN'+\floor{(i+1)N'/p'}+1/2)\gamma-p-ap'-i\bigr)}
  \right)
  \\
  &\times
  \frac{E_N\bigl((N-(a+1)N'+1/2)\gamma-p+(a+1)p'\bigr)}
       {E_N\bigl((N-aN'-\floor{(p'-1)N'/p'}-1/2)\gamma-p+(a+1)p'\bigr)}
  \\
  &\times
  \frac{E_N\bigl((N+aN'+\floor{(p'-1)N'/p'}+1/2)\gamma-p-(a+1)p'+1\bigr)}
       {E_N\bigl((N+(a+1)N'-1/2)\gamma-p-(a+1)p'+1\bigr)}.
\end{align*}
Note that the case where $i=p'-1$ is exceptional.
\par
Using Lemma~\ref{lem:z_z+1} with $z=(N-aN'-\floor{iN'/p'}-1/2)\gamma-p+ap'+i$ ($i=1,\dots,p'-2$) and $z=(N+aN'+\floor{(i+1)N'/p'}-1/2)\gamma-p-ap'-i-1$ ($i=1,\dots,p'-2$), this becomes
\begin{align*}
  &
  \prod_{i=1}^{p'-1}
  \left(
    \left(1-e^{4(p-ap'-i)N\pi^2/\xi}\right)
    \left(1-e^{4(p+ap'+i)N\pi^2/\xi}\right)
  \right)
  \\
  \times&
  \frac{E_N\bigl((N+aN'+1/2)\gamma-p-ap'\bigr)}
       {E_N\bigl((N-aN'-1/2)\gamma-p+ap'+1\bigr)}
  \\
  \times&
  \frac{E_N\bigl((N-(a+1)N'+1/2)\gamma-p+(a+1)p'\bigr)}
       {E_N\bigl((N+(a+1)N'-1/2)\gamma-p-(a+1)p'+1\bigr)}.
\end{align*}
Therefore we have
\begin{align*}
  &\prod_{a=0}^{n-1}\prod_{b=1}^{N'-1}P_{a,b}
  \\
  =&
  \prod_{a=0}^{n-1}
  \prod_{i=1}^{p'-1}
  \left(
    \left(1-e^{4(p-ap'-i)N\pi^2/\xi}\right)
    \left(1-e^{4(p+ap'+i)N\pi^2/\xi}\right)
  \right)
  \\
  &\times
  \prod_{a=0}^{n-1}
  \left(
    \frac{E_N\bigl((N+aN'+1/2)\gamma-p-ap'\bigr)}
         {E_N\bigl((N-aN'-1/2)\gamma-p+ap'+1\bigr)}
  \right.
  \\
  &\phantom{\prod_{a=0}^{n-1}}\quad\times
  \left.
    \frac{E_N\bigl((N-(a+1)N'+1/2)\gamma-p+(a+1)p'\bigr)}
         {E_N\bigl((N+(a+1)N'-1/2)\gamma-p-(a+1)p'+1\bigr)}
  \right)
  \\
  =&
  \frac{1-e^{4pN\pi^2/\xi}}{1-e^{\xi}}
  \prod_{a=0}^{n-1}
  \prod_{i=1}^{p'-1}
  \left(
    \left(1-e^{4(p+ap'+i)N\pi^2/\xi}\right)
    \left(1-e^{4(p-ap'-i)N\pi^2/\xi}\right)
  \right)
  \\
  &\times
  \prod_{a=1}^{n-1}
  \left(
    \frac{1-e^{4(p+ap')N\pi^2/\xi}}{1-e^{(c+a)\xi/c}}
    \times
    \frac{1-e^{4(p-ap')N\pi^2/\xi}}{1-e^{(c-a)\xi/c}}
  \right)
  \\
  &\times
  \frac{E_N\bigl((N-nN'+1/2)\gamma-p+np'\bigr)}
       {E_N\bigl((N+nN'-1/2)\gamma-p-np'+1\bigr)}
  \\
  =&
  \frac{1-e^{4pN\pi^2/\xi}}{1-e^{\xi}}
  \times
  \frac{\prod_{l=1}^{np'-1}\left(1-e^{4(p+l)N\pi^2/\xi}\right)
                           \left(1-e^{4(p-l)N\pi^2/\xi}\right)}
       {\prod_{a=1}^{n-1}\left(1-e^{(c+a)\xi/c}\right)\left(1-e^{(c-a)\xi/c}\right)}
  \\
  &\times
  \frac{E_N\bigl((N-nN'+1/2)\gamma-p+np'\bigr)}
       {E_N\bigl((N+nN'-1/2)\gamma-p-np'+1\bigr)},
\end{align*}
where we use Lemma~\ref{lem:gamma/2} for $w=(N+aN')\gamma-p-ap'$ ($a=0,1,\dots,n-1$) and $w=(N-aN')\gamma-p+ap'$ ($a=1,2,\dots,n-1$) at the second equality.
\par
Similarly, letting $h$ ($0\le h\le p'-1$) be an integer with $hN'/p'<k-nN'<(h+1)N'/p'$, from Corollary~\ref{cor:q_dilog} we have
\begin{align*}
  &\prod_{b=1}^{k-nN'}Q_b
  =
  \prod_{i=0}^{h-1}\left(\prod_{iN'/p'<b<(i+1)N'/p'}Q_b\right)
  \times
  \prod_{hN'/p'<b\le k-nN'}Q_b
  \\
  =&
  \prod_{i=0}^{h-1}
  \left(
    \prod_{iN'/p'<b<(i+1)N'/p'}
    \frac{E_N\bigl((N-nN'-b-1/2)\gamma-p+np'+i+1\bigr)}
         {E_N\bigl((N-nN'-b+1/2)\gamma-p+np'+i+1\bigr)}
  \right.
  \\
  &\phantom{\prod_{i=0}^{h-1}}\quad\times
  \left.
    \prod_{iN'/p'<b<(i+1)N'/p'}
    \frac{E_N\bigl((N+nN'+b-1/2)\gamma-p-np'-i\bigr)}
         {E_N\bigl((N+nN'+b+1/2)\gamma-p-np'-i\bigr)}
  \right)
  \\
  &\times
  \prod_{hN'/p'<b\le k-nN'}
  \frac{E_N\bigl((N-nN'-b-1/2)\gamma-p+np'+h+1\bigr)}
       {E_N\bigl((N-nN'-b+1/2)\gamma-p+np'+h+1\bigr)}
  \\
  &\times
  \prod_{hN'/p'<b\le k-nN'}
  \frac{E_N\bigl((N+nN'+b-1/2)\gamma-p-np'-h\bigr)}
       {E_N\bigl((N+nN'+b+1/2)\gamma-p-np'-h\bigr)}
  \\
  =&
  \prod_{i=0}^{h-1}
  \left(
  \frac{E_N\bigl((N-nN'-\floor{(i+1)N'/p'}-1/2)\gamma-p+np'+i+1\bigr)}
       {E_N\bigl((N-nN'-\floor{iN'/p'}-1/2)\gamma-p+np'+i+1\bigr)}
  \right.
  \\
  &\phantom{\prod_{i=0}^{h-1}}\quad\times
  \left.
  \frac{E_N\bigl((N+nN'+\floor{iN'/p'}+1/2)\gamma-p-np'-i\bigr)}
       {E_N\bigl((N+nN'+\floor{(i+1)N'/p'}+1/2)\gamma-p-np'-i\bigr)}
  \right)
  \\
  &\times
  \frac{E_N\bigl((N-k-1/2)\gamma-p+np'+h+1\bigr)}
       {E_N\bigl((N-nN'-\floor{hN'/p'}-1/2)\gamma-p+np'+h+1\bigr)}
  \\
  &\times
  \frac{E_N\bigl((N+nN'+\floor{hN'/p'}+1/2)\gamma-p-np'-h\bigr)}
       {E_N\bigl((N+k+1/2)\gamma-p-np'-h\bigr)}.
\end{align*}
Using Lemma~\ref{lem:z_z+1} with $z=(N-nN'-\floor{iN'/p'}-1/2)\gamma-p+np'+i$ and $z=(N+nN'+\floor{iN'/p'}+1/2)\gamma-p-np'-i$ ($i=1,2,\dots,h$), we have
\begin{align*}
  \prod_{b=1}^{k-nN'}Q_b
  &=
  \prod_{i=1}^{h}
  \left(
    \left(1-e^{4(p-np'-i)N\pi^2/\xi}\right)
    \left(1-e^{4(p+np'+i)N\pi^2/\xi}\right)
  \right)
  \\
  &\times
  \frac{E_N\bigl((N+nN'+1/2)\gamma-p-np'\bigr)}
       {E_N\bigl((N-nN'-1/2)\gamma-p+np'+1\bigr)}
  \\
  &\times
  \frac{E_N\bigl((N-k-1/2)\gamma-p+np'+h+1\bigr)}
       {E_N\bigl((N+k+1/2)\gamma-p-np'-h\bigr)}.
\end{align*}
\par
Therefore, we finally have
\begin{align*}
  &\prod_{l=1}^{k}\left(1-e^{(N-l)\xi/N}\right)\left(1-e^{(N+l)\xi/N}\right)
  \\
  =&
  \left(1-e^{(c-n)\xi/c}\right)\left(1-e^{(c+n)\xi/c}\right)
  \\
  &\times
  \frac{1-e^{4pN\pi^2/\xi}}{1-e^{\xi}}
  \times
  \prod_{l=1}^{np'-1}\left(1-e^{4(p-l)N\pi^2/\xi}\right)\left(1-e^{4(p+l)N\pi^2/\xi}\right)
  \\
  &\times
  \frac{E_N\bigl((N-nN'+1/2)\gamma-p+np'\bigr)}
       {E_N\bigl((N+nN'-1/2)\gamma-p-np'+1\bigr)}
  \\
  &\times
  \prod_{i=1}^{h}
  \left(
    \left(1-e^{4(p-np'-i)N\pi^2/\xi}\right)
    \left(1-e^{4(p+np'+i)N\pi^2/\xi}\right)
  \right)
  \\
  &\times
  \frac{E_N\bigl((N+nN'+1/2)\gamma-p-np'\bigr)}
       {E_N\bigl((N-nN'-1/2)\gamma-p+np'+1\bigr)}
  \\
  &\times
  \frac{E_N\bigl((N-k-1/2)\gamma-p+np'+h+1\bigr)}
       {E_N\bigl((N+k+1/2)\gamma-p-np'-h\bigr)}
  \\
  =&
  \frac{1-e^{4pN\pi^2/\xi}}{1-e^{\xi}}
  \times
  \prod_{l=1}^{np'}\left(1-e^{4(p-l)N\pi^2/\xi}\right)\left(1-e^{4(p+l)N\pi^2/\xi}\right)
  \\
  &\times
  \prod_{i=1}^{h}
  \left(
    \left(1-e^{4(p-np'-i)N\pi^2/\xi}\right)
    \left(1-e^{4(p+np'+i)N\pi^2/\xi}\right)
  \right)
  \\
  &\times
  \frac{E_N\bigl((N-k-1/2)\gamma-p+np'+h+1\bigr)}
       {E_N\bigl((N+k+1/2)\gamma-p-np'-h\bigr)}
  \\
  =&
  \frac{1-e^{4pN\pi^2/\xi}}{1-e^{\xi}}
  \times
  \prod_{l=1}^{np'+h}
  \left(
    \left(1-e^{4(p-l)N\pi^2/\xi}\right)\left(1-e^{4(p+l)N\pi^2/\xi}\right)
  \right)
  \\
  &\times
  \frac{E_N\bigl((N-k-1/2)\gamma-p+np'+h+1\bigr)}
       {E_N\bigl((N+k+1/2)\gamma-p-np'-h\bigr)}
\end{align*}
where we use Lemma~\ref{lem:gamma/2} for $w=(N-nN')\gamma-p+np'$ and $w=(N+nN')\gamma-p-np'$ at the second equality.
Recalling that we choose $n$ and $h$ so that $nN'\le k<(n+1)N'$ and $hN'/p'<k-nN'<(h+1)N'/p'$, we see that $np'+h$ satisfies $(np'+h)N/p<k<(np'+h+1)N/p$.
So putting $m:=np'+h$ we see that if $mN/p<k<(m+1)N/p$, then the formula above coincides with \eqref{eq:qfac_E_N} where $(p,N)=1$.
\bibliography{mrabbrev,hitoshi}
\bibliographystyle{amsplain}
\end{document}